\documentclass[submission]{FPSAC2017}

\usepackage[utf8x]{inputenc}
\usepackage[english]{babel}
\usepackage{amsmath,amsfonts,amssymb,amsthm}
\usepackage[T1]{fontenc}
\usepackage{dsfont}
\usepackage{mathrsfs}
\usepackage[dvipsnames]{xcolor}
\usepackage{hyperref}
\usepackage{cite}
\usepackage{multicol}
\usepackage{enumitem}
\usepackage{tikz}
\usetikzlibrary{positioning,fit,shapes.geometric}

\title{Combalgebraic structures on decorated cliques}
\author{Samuele Giraudo\addressmark{1}%
    \thanks{Email:
    \href{mailto:samuele.giraudo@u-pem.fr}{samuele.giraudo@u-pem.fr}.}}
\address{\addressmark{1} Université Paris-Est,
    LIGM (UMR $8049$), CNRS, ENPC, ESIEE Paris, UPEM, F-$77454$,
    Marne-la-Vallée, France}
\abstract{
    A new hierarchy of combinatorial operads is introduced, involving
    families of regular polygons with configurations of arcs, called
    decorated cliques. This hierarchy contains, among others, operads on
    noncrossing configurations, Motzkin objects, forests, dissections of
    polygons, and involutions. All this is a consequence of the
    definition of a general functorial construction from unitary magmas
    to operads. We study some of its main properties and show that this
    construction includes the operad of bicolored noncrossing
    configurations and the operads of simple and double multi-tildes. We
    focus in more details on a suboperad of noncrossing decorated
    cliques by computing its presentation, its Koszul dual, and showing
    that it is a Koszul operad.}
\received{\today}

\newtheorem{Theorem}{Theorem}[subsection]
\newtheorem{Proposition}[Theorem]{Proposition}

\numberwithin{equation}{subsection}

\renewcommand{\leq}{\leqslant}
\renewcommand{\geq}{\geqslant}

\newcommand{\N}{\mathbb{N}}
\newcommand{\Z}{\mathbb{Z}}
\newcommand{\K}{\mathbb{K}}
\newcommand{\Mca}{\mathcal{M}}
\newcommand{\Oca}{\mathcal{O}}
\newcommand{\Pfr}{\mathfrak{p}}
\newcommand{\Qfr}{\mathfrak{q}}
\newcommand{\Rfr}{\mathfrak{r}}
\newcommand{\Sfr}{\mathfrak{s}}
\newcommand{\Tfr}{\mathfrak{t}}
\newcommand{\Att}{\mathtt{a}}
\newcommand{\Btt}{\mathtt{b}}

\newcommand{\Hsf}{\mathsf{H}}
\newcommand{\Ksf}{\mathsf{K}}
\newcommand{\Dbb}{\mathbb{D}}
\newcommand{\Unit}{\mathds{1}}
\newcommand{\As}{\mathsf{As}}
\newcommand{\BNC}{\mathsf{BNC}}
\newcommand{\MT}{\mathsf{MT}}
\newcommand{\DMT}{\mathsf{DMT}}
\newcommand{\Op}{\star}
\newcommand{\Arcs}{\mathcal{A}}
\newcommand{\Cli}{\mathrm{C}}
\newcommand{\Gen}{\mathfrak{G}}
\newcommand{\Rel}{\mathfrak{R}}
\newcommand{\OrdBE}{\preceq_{\mathrm{be}}}
\newcommand{\OrdD}{\preceq_{\mathrm{d}}}
\newcommand{\Hamming}{\mathrm{ham}}
\newcommand{\Bub}{\mathrm{Bub}}
\newcommand{\Cro}{\mathrm{Cro}}
\newcommand{\Acy}{\mathrm{Acy}}
\newcommand{\Inf}{\mathrm{Nes}}
\newcommand{\Deg}{\mathrm{Deg}}
\newcommand{\Whi}{\mathrm{Whi}}
\newcommand{\Paths}{\mathrm{Pat}}
\newcommand{\Forests}{\mathrm{For}}
\newcommand{\Motzkin}{\mathrm{Mot}}
\newcommand{\Diss}{\mathrm{Dis}}
\newcommand{\Invol}{\mathrm{Inv}}
\newcommand{\Schro}{\mathrm{Sch}}
\newcommand{\Luc}{\mathsf{Luc}}
\newcommand{\NC}{\mathrm{NC}}
\newcommand{\Triangles}{\mathcal{T}}
\newcommand{\Hilbert}{\mathcal{H}}

\newcommand{\Free}{\mathrm{Free}}
\newcommand{\Del}{\mathrm{d}}
\newcommand{\OEIS}[1]{\href{http://oeis.org/#1}{{\bf #1}}}

\tikzstyle{Node}=[circle,draw=RoyalBlue!80,fill=RoyalBlue!8,inner sep=1pt,
minimum size=2.0mm,thick,font=\scriptsize]
\tikzstyle{Edge}=[draw=BrickRed!80,cap=round,thick]
\tikzstyle{Leaf}=[rectangle,draw=Black!70,fill=Black!16,
inner sep=0pt,minimum size=1mm,thick]
\tikzstyle{EdgeLabel}=[regular polygon,regular polygon sides=6,
draw=ForestGreen!80,fill=ForestGreen!2,inner sep=0pt,line width=.3pt,
minimum size=2.3mm,midway,text=Black,font=\tiny]
\tikzstyle{CliqueEdge}=[draw=Cerulean!90,thick]
\tikzstyle{CliqueEmptyEdge}=[draw=Gray!90,thick,densely dashed]
\tikzstyle{CliqueLabel}=[midway,inner sep=1pt,fill=White!0,
font=\tiny]
\tikzstyle{CliquePoint}=[circle,inner sep=0.75pt,fill=BrickRed!25,
draw=BrickRed!70]
\tikzstyle{Injection}=[Black!100,draw,>->]
\tikzstyle{Surjection}=[Black!100,draw,->>]

\newcommand{\UnitClique}{
\begin{tikzpicture}[scale=.6]
    \node[CliquePoint](1)at(0,0){};
    \node[CliquePoint](2)at(.75,0){};
    \draw[CliqueEmptyEdge](1)edge[]node[CliqueLabel]{}(2);
\end{tikzpicture}}

\begin{document}
\maketitle

\section*{Introduction}
Regular polygons endowed with configurations of arcs are very classical
combinatorial objects. Up to some restrictions or enrichments, these
polygons can be put in bijection with several combinatorial families.
Triangulations are the most celebrated among these, but also noncrossing
configurations~\cite{FN99}, dissections of polygons, noncrossing
partitions, or involutions belong also to this world. As many
combinatorial objects, the polygons of most of these families can be
described by composing or grafting smaller pieces together.
Operads~\cite{LV12,Men15} are algebraic structures abstracting the
notion of planar rooted trees and their grafting operations. For this
reason, operads are one of the most suitable modern algebraic structures
to study such objects. In the last years, a lot of combinatorial sets
have been endowed fruitfully with a structure of an operad (see for
instance~\cite{Cha08,LMN13,CG14,GLMN16,Gir16}), each time providing
results about enumeration, discovering new statistics, or establishing
new links (by morphisms) between different combinatorial sets.

The purpose of this work is twofold. First, we are concerned in endowing
the whole set of polygons with configurations of arcs with a structure
of an operad. This leads to see these objects under a new light,
stressing some of their combinatorial and algebraic properties. Second,
we would provide a general construction of operads of polygons rich
enough so that it includes some already known operads. As a consequence,
we obtain alternative constructions of existing operads and new
interpretations of these. We work here with $\Mca$-decorated cliques,
that are complete graphs whose arcs are labeled on $\Mca$, where $\Mca$
is a unitary magma. These objects are natural generalizations of
polygons with configurations of arcs since the arcs of any
$\Mca$-decorated clique labeled by the unit of $\Mca$ are considered as
missing. The elements of $\Mca$ different from the unit allow moreover
to handle polygons with arcs of different colors. We propose a functor
$\Cli$ from the category of unitary magmas to the category of operads.
It builds, from any unitary magma $\Mca$, an operad $\Cli\Mca$ on
$\Mca$-decorated cliques.

This operad has a lot combinatorial and algebraic properties. First,
$\Cli\Mca$ admits as quotients of operads several structures on
particular families of polygons with configurations of arcs. We can for
instance control the degrees of the vertices, the crossings, or the
nestings between the arcs to obtain new operads. We hence get quotients
of $\Cli\Mca$ involving, among others, Schröder trees, dissections of
polygons, Motzkin objects, forests, with colored versions for each of
these. This leads to a new hierarchy of operads, wherein links between
its elements appear as surjective or injective morphisms of operads (see
Figure~\ref{fig:diagram_operads}).
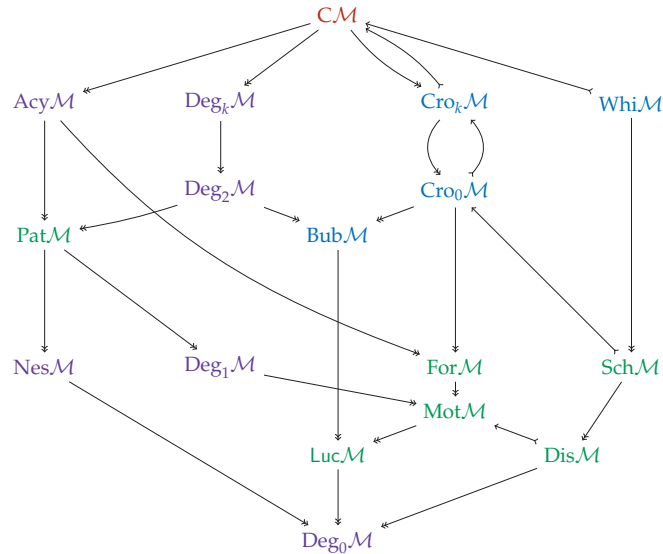
\begin{figure}[ht]
    \centering
    \scalebox{.65}{
    \begin{tikzpicture}[xscale=1.2,yscale=0.9]
        \node[text=BrickRed](CliM)at(8,10)
            {\begin{math}\Cli\Mca\end{math}};
        \node[text=RoyalBlue](CrokM)at(10,8)
            {\begin{math}\Cro_k\Mca\end{math}};
        \node[text=RoyalBlue](Cro0M)at(10,6)
            {\begin{math}\Cro_0\Mca\end{math}};
        \node[text=RoyalPurple](DegkM)at(6,8)
            {\begin{math}\Deg_k\Mca\end{math}};
        \node[text=RoyalPurple](Deg2M)at(6,6)
            {\begin{math}\Deg_2\Mca\end{math}};
        \node[text=RoyalPurple](Deg1M)at(6,2)
            {\begin{math}\Deg_1\Mca\end{math}};
        \node[text=RoyalPurple](Deg0M)at(8,-2)
            {\begin{math}\Deg_0\Mca\end{math}};
        \node[text=RoyalPurple](FinM)at(3,2)
            {\begin{math}\Inf\Mca\end{math}};
        \node[text=RoyalPurple](AcyM)at(3,8)
            {\begin{math}\Acy\Mca\end{math}};
        \node[text=RoyalBlue](BubM)at(8,5)
            {\begin{math}\Bub\Mca\end{math}};
        \node[text=RoyalBlue](WhiM)at(13,8)
            {\begin{math}\Whi\Mca\end{math}};
        \node[text=ForestGreen](PathsM)at(3,5)
            {\begin{math}\Paths\Mca\end{math}};
        \node[text=ForestGreen](LucM)at(8,0)
            {\begin{math}\Luc\Mca\end{math}};
        \node[text=ForestGreen](ForestsM)at(10,2)
            {\begin{math}\Forests\Mca\end{math}};
        \node[text=ForestGreen](SchroM)at(13,2)
            {\begin{math}\Schro\Mca\end{math}};
        \node[text=ForestGreen](MotzkinM)at(10,1)
            {\begin{math}\Motzkin\Mca\end{math}};
        \node[text=ForestGreen](DissM)at(12,0)
            {\begin{math}\Diss\Mca\end{math}};
        \draw[Surjection](CliM)to[bend right=12](CrokM);
        \draw[Injection](CrokM)to[bend right=12](CliM);
        \draw[Injection](WhiM)--(CliM);
        \draw[Surjection](CliM)--(DegkM);
        \draw[Surjection](CliM)--(AcyM);
        \draw[Surjection](AcyM)to[bend right=15](ForestsM);
        \draw[Surjection](AcyM)--(PathsM);
        \draw[Surjection](DegkM)--(Deg2M);
        \draw[Surjection](CrokM)to[bend right=35](Cro0M);
        \draw[Injection](Cro0M)to[bend right=35](CrokM);
        \draw[Surjection](WhiM)--(SchroM);
        \draw[Surjection](Deg2M)to[bend left=5](PathsM);
        \draw[Surjection](Deg2M)--(BubM);
        \draw[Surjection](Cro0M)--(BubM);
        \draw[Surjection](Cro0M)--(ForestsM);
        \draw[Injection](SchroM)--(Cro0M);
        \draw[Surjection](PathsM)--(FinM);
        \draw[Surjection](PathsM)--(Deg1M);
        \draw[Surjection](BubM)--(LucM);
        \draw[Surjection](Deg1M)--(MotzkinM);
        \draw[Surjection](FinM)--(Deg0M);
        \draw[Surjection](ForestsM)--(MotzkinM);
        \draw[Surjection](MotzkinM)--(LucM);
        \draw[Injection](DissM)--(MotzkinM);
        \draw[Surjection](SchroM)--(DissM);
        \draw[Surjection](LucM)--(Deg0M);
        \draw[Surjection](DissM)--(Deg0M);
    \end{tikzpicture}}
    \vspace{-.5em}
    \caption{A diagram of suboperads and quotients of $\Cli\Mca$.
    Arrows~$\rightarrowtail$ (resp.~$\twoheadrightarrow$) are
    injective (resp. surjective) morphisms of operads. Here, $\Mca$ is
    a unitary magma without nontrival unit divisors.}
    \label{fig:diagram_operads}
\end{figure}
One of the most
notable of these is built by considering the $\Mca$-decorated cliques
that have vertices of degrees at most $1$, leading to a quotient
$\Invol\Mca$ of $\Cli\Mca$ involving standard Young tableaux (or
equivalently, involutions). To the best of our  knowledge, $\Invol\Mca$
is the first nontrivial operad on these objects. Besides, the
construction $\Cli$ allows to retrieve the operad $\BNC$ of bicolored
noncrossing configurations~\cite{CG14} and the operads $\MT$ and $\DMT$
respectively defined in~\cite{LMN13} and~\cite{GLMN16} that involve
multi-tildes and double multi-tildes, operators coming from formal
languages theory~\cite{CCM11}. The suboperad $\NC\Mca$ of $\Cli\Mca$ of
$\Mca$-noncrossing configurations, that are $\Mca$-decorated cliques
without crossing diagonals, admits some nice algebraic properties. It is
first generated by elements of arity two (which is not the case of
$\Cli\Mca$), and its nontrivial relations are concentrated in arity
three. This operad is also a Koszul operad.

This text is organized as follows. The construction $\Cli$ is defined in
Section~\ref{sec:construction_clique} and its first properties are
listed. Among other, we describe the generators of $\Cli\Mca$, its
dimensions, establish that it admits a cyclic operad structure, and
define two alternative bases, the $\Hsf$-basis and the $\Ksf$-basis, by
considering a partial order structure on the set of $\Mca$-decorated
cliques. Section~\ref{sec:quotients} is devoted to define some quotients
of $\Cli\Mca$ and to construct, through $\Cli$, the operads $\BNC$,
$\MT$, and $\DMT$. Finally, we study in more details the operad
$\NC\Mca$ in Section~\ref{sec:operads_noncrossing_configurations}. We
show that this operad is an operad of Schröder trees with labels on
arcs satisfying some conditions. We compute its dimensions, its
presentation, its Koszul dual, and establish the fact that it is a
Koszul operad.

\paragraph{Notations and general conventions.}
All the algebraic structures of this article have a field of
characteristic zero $\K$ as ground field. We shall use the classical
notations about operads~\cite{LV12} and more precisely those
of~\cite{Gir16}. Since we consider only nonsymmetric operads, we call
these simply {\em operads}. The sequences of integers cited in the
sequel come from~\cite{Slo}.

\section{Operads of decorated cliques}
\label{sec:construction_clique}

\subsection{Unitary magmas and decorated cliques}
A {\em clique} of {\em size} $n \geq 1$ is a complete graph $\Pfr$ on
the set of vertices $[n + 1]$. An {\em arc} of $\Pfr$ is a pair of
integers $(x, y)$ with $1 \leq x < y \leq n + 1$, a {\em diagonal} is an
arc $(x, y)$ different from $(x, x + 1)$ and $(1, n + 1)$, and an
{\em edge} is an arc of the form $(x, x + 1)$ and different from
$(1, n + 1)$. We denote by $\Arcs_\Pfr$ the set of arcs of $\Pfr$. The
{\em $i$-th edge} of $\Pfr$ is the edge $(i, i + 1)$ and the arc
$(1, n + 1)$ is the {\em base} of $\Pfr$. Let $\Mca$ be a unitary magma,
that is a set endowed with a binary operation $\Op$ admitting a left and
right unit $\Unit_\Mca$. An {\em $\Mca$-decorated clique} (or an
{\em $\Mca$-clique} for short) is a clique $\Pfr$ endowed with a map
$\phi_\Pfr : \Arcs_\Pfr \to \Mca$. For convenience, for any arc $(x, y)$
of $\Pfr$, we shall denote by $\Pfr(x, y)$ the value $\phi_\Pfr((x, y))$.
Moreover, we say that the arc $(x, y)$ is {\em labeled} by $\Pfr(x, y)$.
When the arc $(x, y)$ is labeled by an element different from
$\Unit_\Mca$, we say that the arc $(x, y)$ is {\em solid}.

In our graphical representations, we shall stick to the following
drawing conventions for $\Mca$-cliques. First, each $\Mca$-clique is
depicted so that its base is the bottommost segment and vertices are
implicitly numbered from $1$ to $n + 1$ in the clockwise direction.
Second, the label of any arc $(x, y)$ of $\Pfr$ is represented in the
following way. If $(x, y)$ is solid, we represent it by a line
decorated by $\Pfr(x, y)$. If $(x, y)$ is not solid and is an edge or
the base of $\Pfr$, we represent it as a dashed line. In the remaining
case, when $(x, y)$ is a diagonal of $\Pfr$ and is not solid, we do
not draw it.

To explore some examples in this article, we shall consider the additive
unitary magma $\Z$, the cyclic additive unitary magma $\N_\ell$ on
$\Z/_{\ell \Z}$, and the unitary magma $\Dbb_\ell$ on the set
$\{\Unit, 0, a_1, \dots, a_\ell\}$ where $\Unit$ is the unit of
$\Dbb_\ell$, $0$ is absorbing, and $a_i \Op a_j = 0$ for all
$i, j \in [\ell]$. For instance,
\vspace{-1em}
\begin{equation} \begin{split}\end{split}
    \Pfr :=
    \begin{split}
    \begin{tikzpicture}[scale=.65]
        \node[CliquePoint](1)at(-0.43,-0.90){};
        \node[CliquePoint](2)at(-0.97,-0.22){};
        \node[CliquePoint](3)at(-0.78,0.62){};
        \node[CliquePoint](4)at(-0.00,1.00){};
        \node[CliquePoint](5)at(0.78,0.62){};
        \node[CliquePoint](6)at(0.97,-0.22){};
        \node[CliquePoint](7)at(0.43,-0.90){};
        \draw[CliqueEdge](1)edge[]node[CliqueLabel]
            {\begin{math}-1\end{math}}(2);
        \draw[CliqueEdge](1)edge[]node[CliqueLabel]
            {\begin{math}1\end{math}}(5);
        \draw[CliqueEdge](1)edge[]node[CliqueLabel]
            {\begin{math}2\end{math}}(7);
        \draw[CliqueEmptyEdge](2)edge[]node[CliqueLabel]{}(3);
        \draw[CliqueEmptyEdge](3)edge[]node[CliqueLabel]{}(4);
        \draw[CliqueEdge](3)edge[]node[CliqueLabel,near start]
            {\begin{math}3\end{math}}(7);
        \draw[CliqueEmptyEdge](4)edge[]node[CliqueLabel]{}(5);
        \draw[CliqueEdge](5)edge[]node[CliqueLabel]
            {\begin{math}2\end{math}}(6);
        \draw[CliqueEdge](5)edge[]node[CliqueLabel]
            {\begin{math}2\end{math}}(7);
        \draw[CliqueEmptyEdge](6)edge[]node[CliqueLabel]{}(7);
    \end{tikzpicture}
    \end{split}
\end{equation}
is a $\Z$-clique of size $6$ such that, among others, $\Pfr(1, 2) = -1$,
$\Pfr(1, 5) = 1$, $\Pfr(3, 7) = 3$, $\Pfr(5, 7) = 2$, $\Pfr(2, 3) = 0$,
and~$\Pfr(2, 6) = 0$.

\subsection{A functor from unitary magmas to operads}
For any unitary magma $\Mca$, we define the vector space
\begin{math}
    \Cli\Mca := \bigoplus_{n \geq 1} \Cli\Mca(n)
\end{math}
where $\Cli\Mca(1)$ is the linear span of the singleton consisting in
the $\Mca$-clique
\begin{math} \label{equ:unit_Cli_M}
    \UnitClique
\end{math}
of size $1$ whose base is labeled by $\Unit_\Mca$, and for any
$n \geq 2$, $\Cli\Mca(n)$ is the linear span of all $\Mca$-cliques of
size $n$. We endow $\Cli\Mca$ with a partial composition map $\circ_i$
defined linearly in the following way. If $\Pfr$ and $\Qfr$ are two
$\Mca$-cliques of respective sizes $n$ and $m$, and $i$ is a valid
integer, $\Pfr \circ_i \Qfr$ is obtained by gluing the base of $\Qfr$
onto the $i$-th edge of $\Pfr$, by relabeling the common arcs between
$\Pfr$ and $\Qfr$, respectively the arcs $(i, i + 1)$ and $(1, m + 1)$,
by $\Pfr(i, i + 1) \Op \Qfr(1, m + 1)$, and by renumbering the vertices
of the clique thus obtained from $1$ to $n + m - 1$ (see
Figure~\ref{fig:composition_Cli_M}).
\begin{figure}[ht]\vspace{-1.5em}
    \centering
    \begin{equation*}
        \begin{split}
        \begin{tikzpicture}[scale=.45]
            \node[shape=coordinate](1)at(-0.50,-0.87){};
            \node[shape=coordinate](2)at(-1.00,-0.00){};
            \node[CliquePoint](3)at(-0.50,0.87){};
            \node[CliquePoint](4)at(0.50,0.87){};
            \node[shape=coordinate](5)at(1.00,0.00){};
            \node[shape=coordinate](6)at(0.50,-0.87){};
            \draw[CliqueEdge](1)edge[]node[CliqueLabel]{}(2);
            \draw[CliqueEdge](1)edge[]node[CliqueLabel]{}(6);
            \draw[CliqueEdge](2)edge[]node[CliqueLabel]{}(3);
            \draw[CliqueEdge](3)edge[]node[CliqueLabel]
                {\begin{math}a\end{math}}(4);
            \draw[CliqueEdge](4)edge[]node[CliqueLabel]{}(5);
            \draw[CliqueEdge](5)edge[]node[CliqueLabel]{}(6);
            \node[left of=3,node distance=2mm,font=\scriptsize]
                {\begin{math}i\end{math}};
            \node[right of=4,node distance=4mm,font=\scriptsize]
                {\begin{math}i\!+\!1\end{math}};
            \node[font=\footnotesize](name)at(0,0)
                {\begin{math}\Pfr\end{math}};
        \end{tikzpicture}
        \end{split}
        \enspace \circ_i \enspace\enspace
        \begin{split}
        \begin{tikzpicture}[scale=.45]
            \node[CliquePoint](1)at(-0.50,-0.87){};
            \node[shape=coordinate](2)at(-1.00,-0.00){};
            \node[shape=coordinate](3)at(-0.50,0.87){};
            \node[shape=coordinate](4)at(0.50,0.87){};
            \node[shape=coordinate](5)at(1.00,0.00){};
            \node[CliquePoint](6)at(0.50,-0.87){};
            \draw[CliqueEdge,RawSienna!80]
                (1)edge[]node[CliqueLabel]{}(2);
            \draw[CliqueEdge,draw=RawSienna!80]
                (1)edge[]node[CliqueLabel]{\begin{math}b\end{math}}(6);
            \draw[CliqueEdge,RawSienna!80]
                (2)edge[]node[CliqueLabel]{}(3);
            \draw[CliqueEdge,RawSienna!80]
                (3)edge[]node[CliqueLabel]{}(4);
            \draw[CliqueEdge,RawSienna!80]
                (4)edge[]node[CliqueLabel]{}(5);
            \draw[CliqueEdge,RawSienna!80]
                (5)edge[]node[CliqueLabel]{}(6);
            \node[font=\footnotesize](name)at(0,0)
                {\begin{math}\Qfr\end{math}};
        \end{tikzpicture}
        \end{split}
        \quad = \quad
        \begin{split}
        \begin{tikzpicture}[scale=.45]
            \begin{scope}
            \node[shape=coordinate](1)at(-0.50,-0.87){};
            \node[shape=coordinate](2)at(-1.00,-0.00){};
            \node[CliquePoint](3)at(-0.50,0.87){};
            \node[CliquePoint](4)at(0.50,0.87){};
            \node[shape=coordinate](5)at(1.00,0.00){};
            \node[shape=coordinate](6)at(0.50,-0.87){};
            \draw[CliqueEdge](1)edge[]node[CliqueLabel]{}(2);
            \draw[CliqueEdge](1)edge[]node[CliqueLabel]{}(6);
            \draw[CliqueEdge](2)edge[]node[CliqueLabel]{}(3);
            \draw[CliqueEdge](3)edge[]node[CliqueLabel]
                {\begin{math}a\end{math}}(4);
            \draw[CliqueEdge](4)edge[]node[CliqueLabel]{}(5);
            \draw[CliqueEdge](5)edge[]node[CliqueLabel]{}(6);
            \node[left of=3,node distance=2mm,font=\scriptsize]
                {\begin{math}i\end{math}};
            \node[right of=4,node distance=4mm,font=\scriptsize]
                {\begin{math}i\!+\!1\end{math}};
            \node[font=\footnotesize](name)at(0,0)
                {\begin{math}\Pfr\end{math}};
            \end{scope}
            \begin{scope}[yshift=2.2cm]
            \node[CliquePoint](1)at(-0.50,-0.87){};
            \node[shape=coordinate](2)at(-1.00,-0.00){};
            \node[shape=coordinate](3)at(-0.50,0.87){};
            \node[shape=coordinate](4)at(0.50,0.87){};
            \node[shape=coordinate](5)at(1.00,0.00){};
            \node[CliquePoint](6)at(0.50,-0.87){};
            \draw[CliqueEdge,RawSienna!80]
                (1)edge[]node[CliqueLabel]{}(2);
            \draw[CliqueEdge,draw=RawSienna!80]
                (1)edge[]node[CliqueLabel]{\begin{math}b\end{math}}(6);
            \draw[CliqueEdge,RawSienna!80]
                (2)edge[]node[CliqueLabel]{}(3);
            \draw[CliqueEdge,RawSienna!80]
                (3)edge[]node[CliqueLabel]{}(4);
            \draw[CliqueEdge,RawSienna!80]
                (4)edge[]node[CliqueLabel]{}(5);
            \draw[CliqueEdge,RawSienna!80]
                (5)edge[]node[CliqueLabel]{}(6);
            \node[font=\footnotesize](name)at(0,0)
                {\begin{math}\Qfr\end{math}};
            \end{scope}
        \end{tikzpicture}
        \end{split}
        \quad = \quad
        \begin{split}
        \begin{tikzpicture}[scale=.65]
            \node[shape=coordinate](1)at(-0.31,-0.95){};
            \node[shape=coordinate](2)at(-0.81,-0.59){};
            \node[CliquePoint](3)at(-1.00,-0.00){};
            \node[shape=coordinate](4)at(-0.81,0.59){};
            \node[shape=coordinate](5)at(-0.31,0.95){};
            \node[shape=coordinate](6)at(0.31,0.95){};
            \node[shape=coordinate](7)at(0.81,0.59){};
            \node[CliquePoint](8)at(1.00,0.00){};
            \node[shape=coordinate](9)at(0.81,-0.59){};
            \node[shape=coordinate](10)at(0.31,-0.95){};
            \draw[CliqueEdge](1)edge[]node[]{}(2);
            \draw[CliqueEdge](1)edge[]node[]{}(10);
            \draw[CliqueEdge](2)edge[]node[]{}(3);
            \draw[CliqueEdge,RawSienna!80](3)edge[]node[]{}(4);
            \draw[CliqueEdge,RawSienna!80](4)edge[]node[]{}(5);
            \draw[CliqueEdge,RawSienna!80](5)edge[]node[]{}(6);
            \draw[CliqueEdge,RawSienna!80](6)edge[]node[]{}(7);
            \draw[CliqueEdge,RawSienna!80](7)edge[]node[]{}(8);
            \draw[CliqueEdge](8)edge[]node[]{}(9);
            \draw[CliqueEdge](9)edge[]node[]{}(10);
            \node[left of=3,node distance=2mm,font=\scriptsize]
                {\begin{math}i\end{math}};
            \node[right of=8,node distance=5mm,font=\scriptsize]
                {\begin{math}i\!+\!m\end{math}};
            \draw[CliqueEdge,draw=RoyalBlue!50!RawSienna!50]
                (3)edge[]node[CliqueLabel]
                {\begin{math}a \Op b\end{math}}(8);
        \end{tikzpicture}
        \end{split}
    \end{equation*}
    \vspace{-2em}
    \caption{The partial composition of $\Cli\Mca$. Here, $\Pfr$ and
    $\Qfr$ are two $\Mca$-cliques. The label of the $i$-th edge of
    $\Pfr$ is $a \in \Mca$ and the label of the base of $\Qfr$
    is~$b \in \Mca$. The size of $\Qfr$ is~$m$.}
    \label{fig:composition_Cli_M}
\end{figure}
For example, in $\Cli\Z$, one has the two partial compositions
\begin{equation}
    \begin{split}
    \begin{tikzpicture}[scale=.5]
        \node[CliquePoint](1)at(-0.50,-0.87){};
        \node[CliquePoint](2)at(-1.00,-0.00){};
        \node[CliquePoint](3)at(-0.50,0.87){};
        \node[CliquePoint](4)at(0.50,0.87){};
        \node[CliquePoint](5)at(1.00,0.00){};
        \node[CliquePoint](6)at(0.50,-0.87){};
        \draw[CliqueEdge](1)edge[]node[CliqueLabel]
            {\begin{math}1\end{math}}(2);
        \draw[CliqueEdge](1)edge[]node[CliqueLabel]
            {\begin{math}-2\end{math}}(5);
        \draw[CliqueEmptyEdge](1)edge[]node[CliqueLabel]{}(6);
        \draw[CliqueEdge](2)edge[]node[CliqueLabel]
            {\begin{math}-2\end{math}}(3);
        \draw[CliqueEmptyEdge](3)edge[]node[CliqueLabel]{}(4);
        \draw[CliqueEdge](3)edge[]node[CliqueLabel]
            {\begin{math}1\end{math}}(5);
        \draw[CliqueEmptyEdge](4)edge[]node[CliqueLabel]{}(5);
        \draw[CliqueEmptyEdge](5)edge[]node[CliqueLabel]{}(6);
    \end{tikzpicture}
    \end{split}
    \enspace \circ_2 \enspace
    \begin{split}
    \begin{tikzpicture}[scale=.4]
        \node[CliquePoint](1)at(-0.71,-0.71){};
        \node[CliquePoint](2)at(-0.71,0.71){};
        \node[CliquePoint](3)at(0.71,0.71){};
        \node[CliquePoint](4)at(0.71,-0.71){};
        \draw[CliqueEmptyEdge](1)edge[]node[CliqueLabel]{}(2);
        \draw[CliqueEdge](1)edge[]node[CliqueLabel,near end]
            {\begin{math}1\end{math}}(3);
        \draw[CliqueEdge](1)edge[]node[CliqueLabel]
            {\begin{math}3\end{math}}(4);
        \draw[CliqueEmptyEdge](2)edge[]node[CliqueLabel]{}(3);
        \draw[CliqueEdge](2)edge[]node[CliqueLabel,near start]
            {\begin{math}1\end{math}}(4);
        \draw[CliqueEdge](3)edge[]node[CliqueLabel]
            {\begin{math}2\end{math}}(4);
    \end{tikzpicture}
    \end{split}
    \enspace = \enspace
    \begin{split}
    \begin{tikzpicture}[scale=.7]
        \node[CliquePoint](1)at(-0.38,-0.92){};
        \node[CliquePoint](2)at(-0.92,-0.38){};
        \node[CliquePoint](3)at(-0.92,0.38){};
        \node[CliquePoint](4)at(-0.38,0.92){};
        \node[CliquePoint](5)at(0.38,0.92){};
        \node[CliquePoint](6)at(0.92,0.38){};
        \node[CliquePoint](7)at(0.92,-0.38){};
        \node[CliquePoint](8)at(0.38,-0.92){};
        \draw[CliqueEdge](1)edge[]node[CliqueLabel]
            {\begin{math}1\end{math}}(2);
        \draw[CliqueEdge](1)edge[]node[CliqueLabel]
            {\begin{math}-2\end{math}}(7);
        \draw[CliqueEmptyEdge](1)edge[]node[CliqueLabel]{}(8);
        \draw[CliqueEmptyEdge](2)edge[]node[CliqueLabel]{}(3);
        \draw[CliqueEdge](2)edge[]node[CliqueLabel]
            {\begin{math}1\end{math}}(4);
        \draw[CliqueEdge](2)edge[]node[CliqueLabel]
            {\begin{math}1\end{math}}(5);
        \draw[CliqueEmptyEdge](3)edge[]node[CliqueLabel]{}(4);
        \draw[CliqueEdge](3)edge[]node[CliqueLabel]
            {\begin{math}1\end{math}}(5);
        \draw[CliqueEdge](4)edge[]node[CliqueLabel]
            {\begin{math}2\end{math}}(5);
        \draw[CliqueEmptyEdge](5)edge[]node[CliqueLabel]{}(6);
        \draw[CliqueEdge](5)edge[]node[CliqueLabel]
            {\begin{math}1\end{math}}(7);
        \draw[CliqueEmptyEdge](6)edge[]node[CliqueLabel]{}(7);
        \draw[CliqueEmptyEdge](7)edge[]node[CliqueLabel]{}(8);
    \end{tikzpicture}
    \end{split}\,,
    \qquad \qquad
    \begin{split}
    \begin{tikzpicture}[scale=.5]
        \node[CliquePoint](1)at(-0.50,-0.87){};
        \node[CliquePoint](2)at(-1.00,-0.00){};
        \node[CliquePoint](3)at(-0.50,0.87){};
        \node[CliquePoint](4)at(0.50,0.87){};
        \node[CliquePoint](5)at(1.00,0.00){};
        \node[CliquePoint](6)at(0.50,-0.87){};
        \draw[CliqueEdge](1)edge[]node[CliqueLabel]
            {\begin{math}1\end{math}}(2);
        \draw[CliqueEdge](1)edge[]node[CliqueLabel]
            {\begin{math}-2\end{math}}(5);
        \draw[CliqueEmptyEdge](1)edge[]node[CliqueLabel]{}(6);
        \draw[CliqueEdge](2)edge[]node[CliqueLabel]
            {\begin{math}-2\end{math}}(3);
        \draw[CliqueEmptyEdge](3)edge[]node[CliqueLabel]{}(4);
        \draw[CliqueEdge](3)edge[]node[CliqueLabel]
            {\begin{math}1\end{math}}(5);
        \draw[CliqueEmptyEdge](4)edge[]node[CliqueLabel]{}(5);
        \draw[CliqueEmptyEdge](5)edge[]node[CliqueLabel]{}(6);
    \end{tikzpicture}
    \end{split}
    \enspace \circ_2 \enspace
    \begin{split}
    \begin{tikzpicture}[scale=.4]
        \node[CliquePoint](1)at(-0.71,-0.71){};
        \node[CliquePoint](2)at(-0.71,0.71){};
        \node[CliquePoint](3)at(0.71,0.71){};
        \node[CliquePoint](4)at(0.71,-0.71){};
        \draw[CliqueEmptyEdge](1)edge[]node[CliqueLabel]{}(2);
        \draw[CliqueEdge](1)edge[]node[CliqueLabel,near end]
            {\begin{math}1\end{math}}(3);
        \draw[CliqueEdge](1)edge[]node[CliqueLabel]
            {\begin{math}2\end{math}}(4);
        \draw[CliqueEmptyEdge](2)edge[]node[CliqueLabel]{}(3);
        \draw[CliqueEdge](2)edge[]node[CliqueLabel,near start]
            {\begin{math}1\end{math}}(4);
        \draw[CliqueEdge](3)edge[]node[CliqueLabel]
            {\begin{math}2\end{math}}(4);
    \end{tikzpicture}
    \end{split}
    \enspace = \enspace
    \begin{split}
    \begin{tikzpicture}[scale=.7]
        \node[CliquePoint](1)at(-0.38,-0.92){};
        \node[CliquePoint](2)at(-0.92,-0.38){};
        \node[CliquePoint](3)at(-0.92,0.38){};
        \node[CliquePoint](4)at(-0.38,0.92){};
        \node[CliquePoint](5)at(0.38,0.92){};
        \node[CliquePoint](6)at(0.92,0.38){};
        \node[CliquePoint](7)at(0.92,-0.38){};
        \node[CliquePoint](8)at(0.38,-0.92){};
        \draw[CliqueEdge](1)edge[]node[CliqueLabel]
            {\begin{math}1\end{math}}(2);
        \draw[CliqueEdge](1)edge[]node[CliqueLabel]
            {\begin{math}-2\end{math}}(7);
        \draw[CliqueEmptyEdge](1)edge[]node[CliqueLabel]{}(8);
        \draw[CliqueEmptyEdge](2)edge[]node[CliqueLabel]{}(3);
        \draw[CliqueEdge](2)edge[]node[CliqueLabel]
            {\begin{math}1\end{math}}(4);
        \draw[CliqueEmptyEdge](3)edge[]node[CliqueLabel]{}(4);
        \draw[CliqueEdge](3)edge[]node[CliqueLabel]
            {\begin{math}1\end{math}}(5);
        \draw[CliqueEdge](4)edge[]node[CliqueLabel]
            {\begin{math}2\end{math}}(5);
        \draw[CliqueEmptyEdge](5)edge[]node[CliqueLabel]{}(6);
        \draw[CliqueEdge](5)edge[]node[CliqueLabel]
            {\begin{math}1\end{math}}(7);
        \draw[CliqueEmptyEdge](6)edge[]node[CliqueLabel]{}(7);
        \draw[CliqueEmptyEdge](7)edge[]node[CliqueLabel]{}(8);
    \end{tikzpicture}
    \end{split}\,.
\end{equation}

Moreover, if $\Mca_1$ and $\Mca_2$ are two unitary magmas and
$\phi : \Mca_1 \to \Mca_2$ is a unitary magma morphism, we define
\begin{math}
    \Cli\phi : \Cli\Mca_1 \to \Cli\Mca_2
\end{math}
as the linear map sending any $\Mca_1$-clique $\Pfr$ of size $n$ to the
$\Mca_2$-clique $(\Cli\phi)(\Pfr)$ of size $n$ such that, for any arc
$(x, y) \in \Arcs_\Pfr$,
\begin{math} \label{equ:morphism_Pol_M}
    ((\Cli\phi)(\Pfr))(x, y) := \phi(\Pfr(x, y)).
\end{math}

\begin{Theorem} \label{thm:clique_construction}
    The construction $\Cli$ is a functor from the category of unitary
    magmas to the category of operads. Moreover, $\Cli$ respects
    injections and surjections, and all operads of the image of $\Cli$
    are set-operads.
\end{Theorem}
\begin{proof}
    We just sketch the proof of the fact that $\Cli\Mca$ is an operad
    when $\Mca$ is a unitary magma. This amounts to prove that the
    partial composition of $\Cli\Mca$ satisfies, for all $\Mca$-cliques
    $\Pfr$ (resp. $\Qfr$, $\Rfr$) of sizes $n$ (resp. $m$, $k$),
    \begin{math}
        (\Pfr \circ_i \Qfr) \circ_{i + j - 1} \Rfr
        = \Pfr \circ_i (\Qfr \circ_j \Rfr)
    \end{math}
    where $i \in [n]$, $j \in [m]$,
    \begin{math}
        (\Pfr \circ_i \Qfr) \circ_{j + m - 1} \Rfr
        = (\Pfr \circ_j \Rfr) \circ_i \Qfr
    \end{math}
    where $i < j \in [n]$, and
    \begin{math}
        \UnitClique \circ_1 \Pfr = \Pfr = \Pfr \circ_i \UnitClique
    \end{math}
    where $i \in [n]$. Each of these relations can be checked for
    example with the help of Figure~\ref{fig:composition_Cli_M}.
\end{proof}

\subsection{General properties}

\begin{Proposition} \label{prop:dimensions_Cli_M}
    Let $\Mca$ be a finite unitary magma. For all $n \geq 2$,
    \begin{math} \label{equ:dimensions_Cli_M}
        \dim \Cli\Mca(n) = m^{\binom{n + 1}{2}},
    \end{math}
    where $m := \# \Mca$.
\end{Proposition}

If $\Pfr$ is an $\Mca$-clique, we say that two diagonals $(x, y)$ and
$(x', y')$ of $\Pfr$ are {\em crossed} if $x < x' < y < y'$ or
$x' < x < y' < y$. Let $\Gen_{\Cli\Mca}$ be the set of all
$\Mca$-cliques $\Pfr$ such that, for any diagonal $(x, y)$ of $\Pfr$,
there is at least one {\bf solid} diagonal $(x', y')$ of $\Pfr$ such that
$(x, y)$ and $(x', y')$ are crossed. Observe that, according to this
description, all $\Mca$-cliques of size $2$ belong to~$\Gen_{\Cli\Mca}$.

\begin{Proposition} \label{prop:generating_set_Cli_M}
    Let $\Mca$ be a unitary magma. The set $\Gen_{\Cli\Mca}$ is the
    unique minimal generating set of~$\Cli\Mca$.
\end{Proposition}

Recall that an operad $\Oca$ defined in the category of sets is
{\em basic}~\cite{Val07} if all the maps
$\circ_i^y : \Oca(n) \to \Oca(n + |y| - 1)$, $y \in \Oca$,
defined by $\circ_i(x) := x \circ_i y$ are injective.

\begin{Proposition} \label{prop:basic_Cli_M}
    Let $\Mca$ be a unitary magma. As a set-operad, $\Cli\Mca$ is
    basic if and only if $\Mca$ is right cancellable.
\end{Proposition}

Let $\rho : \Cli\Mca \to \Cli\Mca$ be the linear map sending any
$\Mca$-clique $\Pfr$ to the $\Mca$-clique obtained by rotating by one
step $\Pfr$ in the counterclockwise direction.

\begin{Proposition} \label{prop:cyclic_Cli_M}
    Let $\Mca$ be a unitary magma. The map $\rho$ endows $\Cli\Mca$
    with a cyclic operad structure.
\end{Proposition}

Let $\OrdBE$ (resp. $\OrdD$) be the partial order relation on the set of
all $\Mca$-cliques, where, for any $\Mca$-cliques $\Pfr$ and $\Qfr$, one
has $\Pfr \OrdBE \Qfr$ (resp. $\Pfr \OrdD \Qfr$) if $\Qfr$ can be
obtained from $\Pfr$ by replacing some labels $\Unit_\Mca$ of its edges
or its base (resp. solely of its diagonals) by other labels of $\Mca$.
For any $\Mca$-clique $\Pfr$, let the elements of $\Cli\Mca$ defined by
\begin{math}
    \Hsf_\Pfr :=
    \sum_{\Pfr' \OrdBE \Pfr}
    \Pfr'
\end{math}
and
\begin{math}
    \Ksf_\Pfr :=
    \sum_{\Pfr' \OrdD \Pfr}
    (-1)^{\Hamming(\Pfr', \Pfr)}
    \Pfr',
\end{math}
\noindent where $\Hamming(\Pfr', \Pfr)$ is the {\em Hamming distance}
between $\Pfr'$ and $\Pfr$, that is the number of arcs $(x, y)$ such
that $\Pfr'(x, y) \ne \Pfr(x, y)$. By triangularity and by Möbius
inversion, the family of all the $\Hsf_\Pfr$ (resp. $\Ksf_\Pfr$) forms
a basis of~$\Cli\Mca$, called {\em $\Hsf$-basis} (resp.
{\em $\Ksf$-basis}). For instance, in~$\Cli\Z$,
\begin{equation}\begin{split}\end{split}
    \Hsf_{
    \begin{tikzpicture}[scale=0.4]
        \node[CliquePoint](1)at(-0.59,-0.81){};
        \node[CliquePoint](2)at(-0.95,0.31){};
        \node[CliquePoint](3)at(-0.00,1.00){};
        \node[CliquePoint](4)at(0.95,0.31){};
        \node[CliquePoint](5)at(0.59,-0.81){};
        \draw[CliqueEmptyEdge](1)edge[]node[]{}(2);
        \draw[CliqueEmptyEdge](1)edge[]node[]{}(5);
        \draw[CliqueEmptyEdge](2)edge[]node[]{}(3);
        \draw[CliqueEdge](2)edge[]node[CliqueLabel]
            {\begin{math}1\end{math}}(5);
        \draw[CliqueEdge](3)edge[]node[CliqueLabel]
            {\begin{math}1\end{math}}(4);
        \draw[CliqueEdge](4)edge[]node[CliqueLabel]
            {\begin{math}2\end{math}}(5);
        \draw[CliqueEdge](1)edge[]node[CliqueLabel,near end]
            {\begin{math}2\end{math}}(3);
    \end{tikzpicture}}
    =
    \begin{split}
    \begin{tikzpicture}[scale=0.4]
        \node[CliquePoint](1)at(-0.59,-0.81){};
        \node[CliquePoint](2)at(-0.95,0.31){};
        \node[CliquePoint](3)at(-0.00,1.00){};
        \node[CliquePoint](4)at(0.95,0.31){};
        \node[CliquePoint](5)at(0.59,-0.81){};
        \draw[CliqueEmptyEdge](1)edge[]node[]{}(2);
        \draw[CliqueEmptyEdge](1)edge[]node[]{}(5);
        \draw[CliqueEmptyEdge](2)edge[]node[]{}(3);
        \draw[CliqueEdge](2)edge[]node[CliqueLabel]
            {\begin{math}1\end{math}}(5);
        \draw[CliqueEmptyEdge](3)edge[]node[]{}(4);
        \draw[CliqueEmptyEdge](4)edge[]node[]{}(5);
        \draw[CliqueEdge](1)edge[]node[CliqueLabel,near end]
            {\begin{math}2\end{math}}(3);
    \end{tikzpicture}
    \end{split}
    +
    \begin{split}
    \begin{tikzpicture}[scale=0.4]
        \node[CliquePoint](1)at(-0.59,-0.81){};
        \node[CliquePoint](2)at(-0.95,0.31){};
        \node[CliquePoint](3)at(-0.00,1.00){};
        \node[CliquePoint](4)at(0.95,0.31){};
        \node[CliquePoint](5)at(0.59,-0.81){};
        \draw[CliqueEmptyEdge](1)edge[]node[]{}(2);
        \draw[CliqueEmptyEdge](1)edge[]node[]{}(5);
        \draw[CliqueEmptyEdge](2)edge[]node[]{}(3);
        \draw[CliqueEdge](2)edge[]node[CliqueLabel]
            {\begin{math}1\end{math}}(5);
        \draw[CliqueEmptyEdge](3)edge[]node[]{}(4);
        \draw[CliqueEdge](4)edge[]node[CliqueLabel]
            {\begin{math}2\end{math}}(5);
        \draw[CliqueEdge](1)edge[]node[CliqueLabel,near end]
            {\begin{math}2\end{math}}(3);
    \end{tikzpicture}
    \end{split}
    +
    \begin{split}
    \begin{tikzpicture}[scale=0.4]
        \node[CliquePoint](1)at(-0.59,-0.81){};
        \node[CliquePoint](2)at(-0.95,0.31){};
        \node[CliquePoint](3)at(-0.00,1.00){};
        \node[CliquePoint](4)at(0.95,0.31){};
        \node[CliquePoint](5)at(0.59,-0.81){};
        \draw[CliqueEmptyEdge](1)edge[]node[]{}(2);
        \draw[CliqueEmptyEdge](1)edge[]node[]{}(5);
        \draw[CliqueEmptyEdge](2)edge[]node[]{}(3);
        \draw[CliqueEdge](2)edge[]node[CliqueLabel]
            {\begin{math}1\end{math}}(5);
        \draw[CliqueEdge](3)edge[]node[CliqueLabel]
            {\begin{math}1\end{math}}(4);
        \draw[CliqueEmptyEdge](4)edge[]node[]{}(5);
        \draw[CliqueEdge](1)edge[]node[CliqueLabel,near end]
            {\begin{math}2\end{math}}(3);
    \end{tikzpicture}
    \end{split}
    +
    \begin{split}
    \begin{tikzpicture}[scale=0.4]
        \node[CliquePoint](1)at(-0.59,-0.81){};
        \node[CliquePoint](2)at(-0.95,0.31){};
        \node[CliquePoint](3)at(-0.00,1.00){};
        \node[CliquePoint](4)at(0.95,0.31){};
        \node[CliquePoint](5)at(0.59,-0.81){};
        \draw[CliqueEmptyEdge](1)edge[]node[]{}(2);
        \draw[CliqueEmptyEdge](1)edge[]node[]{}(5);
        \draw[CliqueEmptyEdge](2)edge[]node[]{}(3);
        \draw[CliqueEdge](2)edge[]node[CliqueLabel]
            {\begin{math}1\end{math}}(5);
        \draw[CliqueEdge](3)edge[]node[CliqueLabel]
            {\begin{math}1\end{math}}(4);
        \draw[CliqueEdge](4)edge[]node[CliqueLabel]
            {\begin{math}2\end{math}}(5);
        \draw[CliqueEdge](1)edge[]node[CliqueLabel,near end]
            {\begin{math}2\end{math}}(3);
    \end{tikzpicture}
    \end{split}\,,
    \qquad
    \Ksf_{
    \begin{tikzpicture}[scale=0.4]
        \node[CliquePoint](1)at(-0.59,-0.81){};
        \node[CliquePoint](2)at(-0.95,0.31){};
        \node[CliquePoint](3)at(-0.00,1.00){};
        \node[CliquePoint](4)at(0.95,0.31){};
        \node[CliquePoint](5)at(0.59,-0.81){};
        \draw[CliqueEmptyEdge](1)edge[]node[]{}(2);
        \draw[CliqueEmptyEdge](1)edge[]node[]{}(5);
        \draw[CliqueEmptyEdge](2)edge[]node[]{}(3);
        \draw[CliqueEdge](2)edge[]node[CliqueLabel]
            {\begin{math}1\end{math}}(5);
        \draw[CliqueEdge](3)edge[]node[CliqueLabel]
            {\begin{math}1\end{math}}(4);
        \draw[CliqueEdge](4)edge[]node[CliqueLabel]
            {\begin{math}2\end{math}}(5);
        \draw[CliqueEdge](1)edge[]node[CliqueLabel,near end]
            {\begin{math}2\end{math}}(3);
    \end{tikzpicture}}
    =
    \begin{split}
    \begin{tikzpicture}[scale=0.4]
        \node[CliquePoint](1)at(-0.59,-0.81){};
        \node[CliquePoint](2)at(-0.95,0.31){};
        \node[CliquePoint](3)at(-0.00,1.00){};
        \node[CliquePoint](4)at(0.95,0.31){};
        \node[CliquePoint](5)at(0.59,-0.81){};
        \draw[CliqueEmptyEdge](1)edge[]node[]{}(2);
        \draw[CliqueEmptyEdge](1)edge[]node[]{}(5);
        \draw[CliqueEmptyEdge](2)edge[]node[]{}(3);
        \draw[CliqueEdge](2)edge[]node[CliqueLabel]
            {\begin{math}1\end{math}}(5);
        \draw[CliqueEdge](3)edge[]node[CliqueLabel]
            {\begin{math}1\end{math}}(4);
        \draw[CliqueEdge](4)edge[]node[CliqueLabel]
            {\begin{math}2\end{math}}(5);
        \draw[CliqueEdge](1)edge[]node[CliqueLabel,near end]
            {\begin{math}2\end{math}}(3);
    \end{tikzpicture}
    \end{split}
    -
    \begin{split}
    \begin{tikzpicture}[scale=0.4]
        \node[CliquePoint](1)at(-0.59,-0.81){};
        \node[CliquePoint](2)at(-0.95,0.31){};
        \node[CliquePoint](3)at(-0.00,1.00){};
        \node[CliquePoint](4)at(0.95,0.31){};
        \node[CliquePoint](5)at(0.59,-0.81){};
        \draw[CliqueEmptyEdge](1)edge[]node[]{}(2);
        \draw[CliqueEmptyEdge](1)edge[]node[]{}(5);
        \draw[CliqueEmptyEdge](2)edge[]node[]{}(3);
        \draw[CliqueEdge](3)edge[]node[CliqueLabel]
            {\begin{math}1\end{math}}(4);
        \draw[CliqueEdge](4)edge[]node[CliqueLabel]
            {\begin{math}2\end{math}}(5);
        \draw[CliqueEdge](1)edge[]node[CliqueLabel,near end]
            {\begin{math}2\end{math}}(3);
    \end{tikzpicture}
    \end{split}
    -
    \begin{split}
    \begin{tikzpicture}[scale=0.4]
        \node[CliquePoint](1)at(-0.59,-0.81){};
        \node[CliquePoint](2)at(-0.95,0.31){};
        \node[CliquePoint](3)at(-0.00,1.00){};
        \node[CliquePoint](4)at(0.95,0.31){};
        \node[CliquePoint](5)at(0.59,-0.81){};
        \draw[CliqueEmptyEdge](1)edge[]node[]{}(2);
        \draw[CliqueEmptyEdge](1)edge[]node[]{}(5);
        \draw[CliqueEmptyEdge](2)edge[]node[]{}(3);
        \draw[CliqueEdge](2)edge[]node[CliqueLabel]
            {\begin{math}1\end{math}}(5);
        \draw[CliqueEdge](3)edge[]node[CliqueLabel]
            {\begin{math}1\end{math}}(4);
        \draw[CliqueEdge](4)edge[]node[CliqueLabel]
            {\begin{math}2\end{math}}(5);
    \end{tikzpicture}
    \end{split}
    +
    \begin{split}
    \begin{tikzpicture}[scale=0.4]
        \node[CliquePoint](1)at(-0.59,-0.81){};
        \node[CliquePoint](2)at(-0.95,0.31){};
        \node[CliquePoint](3)at(-0.00,1.00){};
        \node[CliquePoint](4)at(0.95,0.31){};
        \node[CliquePoint](5)at(0.59,-0.81){};
        \draw[CliqueEmptyEdge](1)edge[]node[]{}(2);
        \draw[CliqueEmptyEdge](1)edge[]node[]{}(5);
        \draw[CliqueEmptyEdge](2)edge[]node[]{}(3);
        \draw[CliqueEdge](3)edge[]node[CliqueLabel]
            {\begin{math}1\end{math}}(4);
        \draw[CliqueEdge](4)edge[]node[CliqueLabel]
            {\begin{math}2\end{math}}(5);
    \end{tikzpicture}
    \end{split}\,.
\end{equation}

If $\Pfr$ is an $\Mca$-clique, we denote by $\Pfr_0$ (resp. $\Pfr_i$)
the label of its base (resp. $i$-th edge). Moreover,
$\Del_0(\Pfr)$ (resp. $\Del_i(\Pfr)$) is
the $\Mca$-clique obtained by replacing the label of the base
(resp. $i$-th edge) of $\Pfr$ by $\Unit_\Mca$.

\begin{Proposition} \label{prop:composition_Pol_M_basis_H}
    Let $\Mca$ be a unitary magma. The partial composition of $\Cli\Mca$
    expresses over the $\Hsf$-basis, for any $\Mca$-cliques
    $\Pfr$ and $\Qfr$ different from $\UnitClique$ and any valid integer
    $i$, as
    \begin{equation}
        \Hsf_\Pfr \circ_i \Hsf_\Qfr
        =
        \begin{cases}
            \Hsf_{\Pfr \circ_i \Qfr}
            + \Hsf_{\Del_i(\Pfr) \circ_i \Qfr}
            + \Hsf_{\Pfr \circ_i \Del_0(\Qfr)}
            + \Hsf_{\Del_i(\Pfr) \circ_i \Del_0(\Qfr)}
                & \mbox{if } \Pfr_i \ne \Unit_\Mca \mbox{ and }
                    \Qfr_0 \ne \Unit_\Mca, \\
            \Hsf_{\Pfr \circ_i \Qfr}
            + \Hsf_{\Del_i(\Pfr) \circ_i \Qfr}
                & \mbox{if } \Pfr_i \ne \Unit_\Mca, \\
            \Hsf_{\Pfr \circ_i \Qfr}
            + \Hsf_{\Pfr \circ_i \Del_0(\Qfr)}
                & \mbox{if } \Qfr_0 \ne \Unit_\Mca, \\
            \Hsf_{\Pfr \circ_i \Qfr} & \mbox{otherwise}.
        \end{cases}
    \end{equation}
\end{Proposition}

\begin{Proposition} \label{prop:composition_Pol_M_basis_K}
    Let $\Mca$ be a unitary magma. The partial composition of $\Cli\Mca$
    expresses over the $\Ksf$-basis, for any $\Mca$-cliques
    $\Pfr$ and $\Qfr$ different from $\UnitClique$ and any valid integer
    $i$, as
    \begin{equation}
        \Ksf_\Pfr \circ_i \Ksf_\Qfr
        =
        \begin{cases}
            \Ksf_{\Pfr \circ_i \Qfr}
                & \mbox{if }
                \Pfr_i \Op \Qfr_0 = \Unit_\Mca, \\
            \Ksf_{\Pfr \circ_i \Qfr} +
            \Ksf_{\Del_i(\Pfr) \circ_i \Del_0(\Qfr)}
                & \mbox{otherwise}.
        \end{cases}
    \end{equation}
\end{Proposition}

For instance, in $\Cli\Z$,
\begin{equation}
    \Hsf_{
    \begin{tikzpicture}[scale=0.25]
        \node[CliquePoint](1)at(-0.87,-0.50){};
        \node[CliquePoint](2)at(-0.00,1.00){};
        \node[CliquePoint](3)at(0.87,-0.50){};
        \draw[CliqueEmptyEdge](1)edge[]node[]{}(2);
        \draw[CliqueEmptyEdge](1)edge[]node[]{}(3);
        \draw[CliqueEdge](2)edge[]node[CliqueLabel]
            {\begin{math}1\end{math}}(3);
    \end{tikzpicture}}
    \circ_2
    \Hsf_{
    \begin{tikzpicture}[scale=0.25]
        \node[CliquePoint](1)at(-0.87,-0.50){};
        \node[CliquePoint](2)at(-0.00,1.00){};
        \node[CliquePoint](3)at(0.87,-0.50){};
        \draw[CliqueEmptyEdge](1)edge[]node[]{}(2);
        \draw[CliqueEdge](1)edge[]node[CliqueLabel]
            {\begin{math}1\end{math}}(3);
        \draw[CliqueEmptyEdge](2)edge[]node[]{}(3);
    \end{tikzpicture}}
    =
    \Hsf_{
    \begin{tikzpicture}[scale=0.35]
        \node[CliquePoint](1)at(-0.71,-0.71){};
        \node[CliquePoint](2)at(-0.71,0.71){};
        \node[CliquePoint](3)at(0.71,0.71){};
        \node[CliquePoint](4)at(0.71,-0.71){};
        \draw[CliqueEmptyEdge](1)edge[]node[]{}(2);
        \draw[CliqueEmptyEdge](1)edge[]node[]{}(4);
        \draw[CliqueEmptyEdge](2)edge[]node[]{}(3);
        \draw[CliqueEmptyEdge](3)edge[]node[]{}(4);
    \end{tikzpicture}}
    +
    2\;
    \Hsf_{
    \begin{tikzpicture}[scale=0.35]
        \node[CliquePoint](1)at(-0.71,-0.71){};
        \node[CliquePoint](2)at(-0.71,0.71){};
        \node[CliquePoint](3)at(0.71,0.71){};
        \node[CliquePoint](4)at(0.71,-0.71){};
        \draw[CliqueEmptyEdge](1)edge[]node[]{}(2);
        \draw[CliqueEmptyEdge](1)edge[]node[]{}(4);
        \draw[CliqueEmptyEdge](2)edge[]node[]{}(3);
        \draw[CliqueEdge](2)edge[]node[CliqueLabel]
            {\begin{math}1\end{math}}(4);
        \draw[CliqueEmptyEdge](3)edge[]node[]{}(4);
    \end{tikzpicture}}
    +
    \Hsf_{
    \begin{tikzpicture}[scale=0.35]
        \node[CliquePoint](1)at(-0.71,-0.71){};
        \node[CliquePoint](2)at(-0.71,0.71){};
        \node[CliquePoint](3)at(0.71,0.71){};
        \node[CliquePoint](4)at(0.71,-0.71){};
        \draw[CliqueEmptyEdge](1)edge[]node[]{}(2);
        \draw[CliqueEmptyEdge](1)edge[]node[]{}(4);
        \draw[CliqueEmptyEdge](2)edge[]node[]{}(3);
        \draw[CliqueEdge](2)edge[]node[CliqueLabel]
            {\begin{math}2\end{math}}(4);
        \draw[CliqueEmptyEdge](3)edge[]node[]{}(4);
    \end{tikzpicture}}\,,
    \qquad
    \Ksf_{
    \begin{tikzpicture}[scale=0.25]
        \node[CliquePoint](1)at(-0.87,-0.50){};
        \node[CliquePoint](2)at(-0.00,1.00){};
        \node[CliquePoint](3)at(0.87,-0.50){};
        \draw[CliqueEmptyEdge](1)edge[]node[]{}(2);
        \draw[CliqueEmptyEdge](1)edge[]node[]{}(3);
        \draw[CliqueEdge](2)edge[]node[CliqueLabel]
            {\begin{math}1\end{math}}(3);
    \end{tikzpicture}}
    \circ_2
    \Ksf_{
    \begin{tikzpicture}[scale=0.25]
        \node[CliquePoint](1)at(-0.87,-0.50){};
        \node[CliquePoint](2)at(-0.00,1.00){};
        \node[CliquePoint](3)at(0.87,-0.50){};
        \draw[CliqueEmptyEdge](1)edge[]node[]{}(2);
        \draw[CliqueEdge](1)edge[]node[CliqueLabel]
            {\begin{math}1\end{math}}(3);
        \draw[CliqueEmptyEdge](2)edge[]node[]{}(3);
    \end{tikzpicture}}
    =
    \Ksf_{
    \begin{tikzpicture}[scale=0.35]
        \node[CliquePoint](1)at(-0.71,-0.71){};
        \node[CliquePoint](2)at(-0.71,0.71){};
        \node[CliquePoint](3)at(0.71,0.71){};
        \node[CliquePoint](4)at(0.71,-0.71){};
        \draw[CliqueEmptyEdge](1)edge[]node[]{}(2);
        \draw[CliqueEmptyEdge](1)edge[]node[]{}(4);
        \draw[CliqueEmptyEdge](2)edge[]node[]{}(3);
        \draw[CliqueEmptyEdge](3)edge[]node[]{}(4);
    \end{tikzpicture}}
    +
    \Ksf_{
    \begin{tikzpicture}[scale=0.35]
        \node[CliquePoint](1)at(-0.71,-0.71){};
        \node[CliquePoint](2)at(-0.71,0.71){};
        \node[CliquePoint](3)at(0.71,0.71){};
        \node[CliquePoint](4)at(0.71,-0.71){};
        \draw[CliqueEmptyEdge](1)edge[]node[]{}(2);
        \draw[CliqueEmptyEdge](1)edge[]node[]{}(4);
        \draw[CliqueEmptyEdge](2)edge[]node[]{}(3);
        \draw[CliqueEdge](2)edge[]node[CliqueLabel]
            {\begin{math}2\end{math}}(4);
        \draw[CliqueEmptyEdge](3)edge[]node[]{}(4);
    \end{tikzpicture}}\,.
\end{equation}

\section{Quotients and suboperads}
\label{sec:quotients}

\subsection{Operads on subfamilies of $\Mca$-cliques}
\label{subsec:quotients}

We now define quotients of $\Cli\Mca$, leading to the construction of
some new operads involving various combinatorial objects which are,
basically, $\Mca$-cliques with some restrictions.
Figure~\ref{fig:diagram_operads} shows a diagram containing all the
considered quotients and suboperads of~$\Cli\Mca$.

\begin{description}[fullwidth]
\item[Bubbles.]
An $\Mca$-clique is an {\em $\Mca$-bubble} if it has no solid
diagonals. Let $\Rel_{\Bub\Mca}$ be the subspace of $\Cli\Mca$
generated by all $\Mca$-cliques that are not bubbles. As quotient of
vector spaces,
\begin{math}
    \Bub\Mca := \Cli\Mca/_{\Rel_{\Bub\Mca}}
\end{math}
is the linear span of all $\Mca$-bubbles. Moreover, the space
$\Bub\Mca$ is a quotient of operads of $\Cli\Mca$. When $\Mca$ is
finite, the dimensions of $\Bub \Mca$ satisfy, for any $n \geq 2$,
\begin{math}
    \dim \Bub\Mca(n) = m^{n + 1},
\end{math}
where $m := \# \Mca$.
\item[White $\Mca$-cliques.]
An $\Mca$-clique is {\em white} if it has no solid edges nor
solid base. Let $\Whi\Mca$ be the subspace of $\Cli\Mca$ of all white
$\Mca$-cliques. The space $\Whi\Mca$ is a suboperad of~$\Cli\Mca$.
When $\Mca$ is finite, the dimensions of $\Whi\Mca$ satisfy, for any
$n \geq 2$,
\begin{math}
    \dim \Whi\Mca(n) =
    m^{(n + 1)(n - 2) / 2},
\end{math}
where $m := \# \Mca$.
\item[Restricting the crossing.]
The {\em crossing} of a solid diagonal of an $\Mca$-clique $\Pfr$
is the number of solid diagonals crossing it. The {\em crossing} of
$\Pfr$ is the maximal crossing of its solid diagonals. For any integer
$k \geq 0$, let $\Rel_{\Cro_k\Mca}$ be the subspace of $\Cli\Mca$
generated by all $\Mca$-cliques of crossings greater than $k$. As
quotient of vector spaces,
\begin{math}
    \Cro_k\Mca := \Cli\Mca/_{\Rel_{\Cro_k\Mca}}
\end{math}
is the linear span of all $\Mca$-cliques of crossings no greater
than~$k$. Moreover, the space $\Cro_k\Mca$ is both a quotient and a
suboperad of~$\Cli\Mca$.
Observe that $\Cro_0\Mca$ is the operad $\Bub\Mca$. Let us set
$\NC\Mca := \Cro_0\Mca$. Any $\Mca$-clique of $\NC\Mca$ is a noncrossing
configuration~\cite{FN99} where each diagonal is decorated by an element
of $\Mca \setminus \{\Unit_\Mca\}$. These operads $\NC\Mca$ have a lot
of nice properties and will be studied in
Section~\ref{sec:operads_noncrossing_configurations}.
\item[Acyclic $\Mca$-cliques.]
An $\Mca$-clique is {\em acyclic} if it does not contain any cycle
formed by solid arcs. Let $\Rel_{\Acy\Mca}$ be the subspace of
$\Cli\Mca$ generated by all $\Mca$-cliques that are not acyclic. As
quotient of vector spaces,
\begin{math}
    \Acy\Mca := \Cli\Mca/_{\Rel_{\Acy\Mca}}
\end{math}
is the linear span of all acyclic $\Mca$-cliques. When $\Mca$ has no
nontrivial unit divisors, the space $\Acy\Mca$ is a quotient of
operads of $\Cli\Mca$. Any $\Dbb_0$-clique of $\Acy\Dbb_0$ can be seen
as a forest of trees. The dimensions of this operad begin by $1$, $7$,
$38$, $291$, $2932$ (Sequence~\OEIS{A001858}, except for the first
terms).
\item[Nesting-free $\Mca$-cliques.]
A solid arc $(x', y')$ is {\em nested} in a solid arc $(x, y)$
of an $\Mca$-clique $\Pfr$ if $x \leq x' < y' \leq y$. We say that
$\Pfr$ is {\em nesting-free} if for any solid arcs $(x, y)$ and
$(x', y')$ of $\Pfr$ such that $(x', y')$ is nested in $(x, y)$,
$(x', y') = (x, y)$. Let $\Rel_{\Inf\Mca}$ be the subspace of
$\Cli\Mca$ generated by all $\Mca$-cliques that are not nesting-free.
As quotient of vector spaces,
\begin{math}
    \Inf\Mca := \Cli\Mca/_{\Rel_{\Inf\Mca}}
\end{math}
is the linear span of all nesting-free $\Mca$-cliques. When $\Mca$ has
no nontrivial unit divisors, the space $\Inf\Mca$ is a quotient of
operads of $\Cli\Mca$. Any $\Dbb_0$-clique of $\Inf\Dbb_0$ can be seen
as an nesting-free clique. The dimensions of this operad begin by $1$,
$5$, $14$, $42$, $132$, and are Catalan numbers
(Sequence~\OEIS{A000108}, except for the first terms). In the same way
as considering $\Mca$-cliques of crossings no greater than $k$ leads
to quotients $\Cro_k\Mca$ of $\Cli\Mca$, it is possible to define
analogous quotients $\Inf_k\Mca$ spanned by $\Mca$-cliques having
solid arcs that nest at most $k$ other ones.
\item[Restricting the degree.]
The {\em degree} of a vertex $x$ of an $\Mca$-clique $\Pfr$ is
the number of solid arcs adjacent to $x$. The {\em degree} of $\Pfr$
is the maximal degree of its vertices. For any integer $k \geq 0$, let
$\Rel_{\Deg_k\Mca}$ be the subspace of $\Cli\Mca$ generated by all
$\Mca$-cliques of degrees greater than $k$. As quotient of vector spaces,
\begin{math}
    \Deg_k\Mca := \Cli\Mca/_{\Rel_{\Deg_k\Mca}}
\end{math}
is the linear span of all $\Mca$-cliques of degrees no greater than~$k$.
When $\Mca$ has  no nontrivial unit divisors, the space $\Deg_k\Mca$
is a quotient of operads of $\Cli\Mca$. Observe that $\Deg_0\Mca$ is
the associative operad $\As$. Let us set $\Invol\Mca := \Deg_1\Mca$.
Any $\Dbb_0$-clique of $\Invol\Dbb_0$ of size $n$ can be seen as a
partition of the set $[n + 1]$ in singletons or pairs. In this case,
$\Invol\Dbb_0$ involves involutions, or equivalently standard Young
tableaux. The dimensions of this operad begin by $1$, $4$, $10$, $26$,
$76$ (Sequence~\OEIS{A000085}, except for the first terms). Moreover,
the dimensions of $\Invol\Dbb_1$ begin by $1$, $7$, $25$, $81$, $331$
(Sequence~\OEIS{A047974}, except for the first terms). Besides, any
$\Dbb_0$-clique of $\Deg_2\Dbb_0$ can be seen as a {\em thunderstorm
graph} ({\em i.e.}, a graph where connected components are cycles or
paths). The dimensions of this operad begin by $1$, $8$, $41$, $253$,
$1858$ (Sequence~\OEIS{A136281}, except for the first terms).
\end{description}

\subsection{Mixing quotients and substructures}
For any operad $\Oca$ and ideals of operads $\Rel_1$ and $\Rel_2$ of
$\Oca$, the space $\Rel_1 + \Rel_2$ is still an ideal of operads of
$\Oca$, and $\Oca/_{\Rel_1 + \Rel_2}$ is a quotient of operads of both
$\Oca/_{\Rel_1}$ and $\Oca/_{\Rel_2}$. Moreover, if $\Oca'$ is a
suboperad of $\Oca$ and $\Rel$ is an ideal of operads of $\Oca$, the
space $\Rel \cap \Oca'$ is an ideal of operads of $\Oca'$, and
$\Oca'/_{\Rel \cap \Oca'}$ is a quotient of operads of $\Oca'$. For these
reasons, we can combine the constructions of
Section~\ref{subsec:quotients} to build a bunch of new quotients of
operads of~$\Cli\Mca$.

When $\Mca$ is finite and has cardinal $2$, several interesting
phenomenons occur already. In this case, $\Mca$ is necessarily isomorphic
to $\N_2$ or to $\Dbb_0$, but only $\Dbb_0$ satisfies the conditions
required by all the propositions of Section~\ref{subsec:quotients}. The
obtained substructures of $\Cli\Dbb_0$ are operads that involve some
very classical combinatorial objects. For instance:
\begin{description}[fullwidth]
\item[Schröder trees.]
Let $\Schro\Mca := \Whi\Mca/_{\Rel_{\Cro_0\Mca} \cap \Whi\Mca}$.
The operad $\Schro\Dbb_0$ involves Schröder trees. Its
dimensions begin by $1$, $1$, $3$, $11$, $45$
(Sequence~\OEIS{A001003}).
\item[Forests of paths.]
Let $\Paths\Mca := \Cli\Mca/_{\Rel_{\Acy\Mca} + \Rel_{\Deg_2\Mca}}$.
The operad $\Paths\Dbb_0$
involves forests of non-rooted trees that are paths. Its
dimensions begin by $1$, $7$, $34$, $206$, $1486$
(Sequence~\OEIS{A011800}, except for the first terms).
\item[Forests of trees.]
Let $\Forests\Mca := \Cli\Mca/_{\Rel_{\Acy\Mca} + \Rel_{\Cro_0\Mca}}$.
The operad $\Forests\Dbb_0$
involves forests of rooted trees without crossing edges. Its
dimensions begin by $1$, $7$, $33$, $181$, $1083$
(Sequence~\OEIS{A054727}, except for the first terms).
\item[Motzkin configurations.]
Let $\Motzkin\Mca := \Cli\Mca/_{\Rel_{\Cro_0\Mca} + \Rel_{\Deg_1\Mca}}$.
The operad $\Motzkin\Dbb_0$
involves Motzkin paths. Its dimensions begin by $1$, $4$, $9$, $21$,
$51$ (Sequence~\OEIS{A001006}, except for the first terms).
\item[Dissections of polygons.]
Let $\Diss\Mca := \Whi\Mca/_{(\Rel_{\Cro_0\Mca} + \Rel_{\Deg_1\Mca})
\cap \Whi\Mca}$.
The operad $\Diss\Dbb_0$ involves dissections of polygons by strictly
disjoint diagonals. Its dimensions begin by $1$, $1$, $3$, $6$, $13$
(Sequence~\OEIS{A093128}, except for the first terms).
\item[Lucas configurations.]
Let $\Luc\Mca := \Cli\Mca/_{\Rel_{\Bub\Mca} + \Rel_{\Deg_1\Mca}}$.
The operad $\Luc\Dbb_0$ involves $\Dbb_0$-bubbles $\Pfr$ such that
any vertex of $\Pfr$ belongs to at most one solid edge. Its
dimensions begin by $1$, $4$, $7$, $11$, $18$, and are Lucas numbers
(Sequence~\OEIS{A000032}, except for the first terms).
\end{description}

\subsection{Constructing existing operads}
\label{subsec:constructing_existing_operads}
We give here three examples of already known operads that can be build
through the construction~$\Cli$.

\paragraph{Bicolored noncrossing configurations.}
The {\em operad of bicolored noncrossing configurations $\BNC$} is an
operad defined in~\cite{CG14} which involves noncrossing configurations
where each solid diagonal can be blue or red, and each edge can be blue
or uncolored. This operad is in fact a special case of our general
construction $\Cli$. Let $\Mca_\BNC := \{\Unit, \Att, \Btt\}$ be the
unitary magma wherein operation $\Op$ is defined so that $\Att$ and
$\Btt$ are idempotent, and $\Att \Op \Btt = \Unit = \Btt \Op \Att$.
Observe that $\Mca_{\BNC}$ is a commutative unitary magma, but, since
$(\Btt \Op \Att) \Op \Att = \Att$ and $\Btt \Op (\Att \Op \Att) = \Unit$,
the operation $\Op$ is not associative.

\begin{Proposition} \label{prop:construction_BNC}
    The suboperad of $\NC\Mca_\BNC$ consisting in its unit and all
    $\Mca_\BNC$-noncrossing configurations without edges labeled by
    $\Unit$ is isomorphic to the operad~$\BNC$.
\end{Proposition}

\paragraph{Multi-tildes and double multi-tildes.}
Appearing from the context of formal languages theory, {\em multi-tildes}
are operators introduced in~\cite{CCM11} as tools offering a convenient
way to express regular languages. As shown in~\cite{LMN13}, the set of
all multi-tildes admits a very natural structure of an operad~$\MT$.
{\em Double multi-tildes} are extensions of these operators introduced
in~\cite{GLMN16} that increase their expressiveness and admit also a
structure of an operad~$\DMT$. Let $\Mca_\DMT$ be the unitary magma
$\Mca_\DMT := \Dbb_0^2$ and $\Mca_\MT$ be the sub-unitary magma of
$\Mca_\DMT$ on the set $\{(\Unit, \Unit), (0, \Unit)\}$.

\begin{Proposition} \label{prop:construction_MT_DMT}
    The operad $\Cli\Mca_\DMT$ (resp. $\Cli\Mca_\MT$) is isomorphic to
    the suboperad of $\DMT$ (resp. $\MT$) consisting in all double
    (resp. simple) multi-tildes except the three (resp. one) nontrivial
    ones of arity~$1$.
\end{Proposition}

\section{Operads of decorated noncrossing configurations}
\label{sec:operads_noncrossing_configurations}
In this section, we study in details the suboperad $\NC\Mca$ of
$\Cli\Mca$. As observed in Section~\ref{subsec:quotients}, this operad
involves all $\Mca$-cliques that do not admit crossing solid diagonals.
We call {\em $\Mca$-noncrossing configurations} such objects.

\subsection{General properties}
The set of all $\Mca$-noncrossing configurations is in one-to-one
correspondence with the set of Schröder trees ({\em i.e.}, rooted planar
trees where internal nodes have arities $2$ or more) where the edges
adjacent to the roots are labeled on $\Mca$, the edges connecting two
internal nodes are labeled on $\Mca \setminus \{\Unit_\Mca\}$, and the
edges adjacent to the leaves are labeled on $\Mca$. This is realized by
computing the dual trees of $\Mca$-noncrossing configurations by
considering the labels of the solid diagonals. We call these trees
{\em $\Mca$-dual trees}. Here is an example of a $\Z$-noncrossing
configuration and the $\Z$-dual tree encoding it:
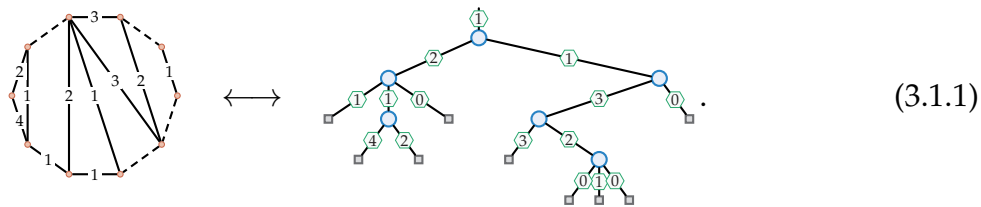
\vspace{-1em}
\begin{equation}
    \begin{split}
    \begin{tikzpicture}[scale=1.1]
        \node[CliquePoint](1)at(-0.31,-0.95){};
        \node[CliquePoint](2)at(-0.81,-0.59){};
        \node[CliquePoint](3)at(-1.00,-0.00){};
        \node[CliquePoint](4)at(-0.81,0.59){};
        \node[CliquePoint](5)at(-0.31,0.95){};
        \node[CliquePoint](6)at(0.31,0.95){};
        \node[CliquePoint](7)at(0.81,0.59){};
        \node[CliquePoint](8)at(1.00,0.00){};
        \node[CliquePoint](9)at(0.81,-0.59){};
        \node[CliquePoint](10)at(0.31,-0.95){};
        \draw[CliqueEdge](1)edge[]node[CliqueLabel]
            {\begin{math}1\end{math}}(2);
        \draw[CliqueEdge](1)edge[]node[CliqueLabel]
            {\begin{math}2\end{math}}(5);
        \draw[CliqueEdge](1)edge[]node[CliqueLabel]
            {\begin{math}1\end{math}}(10);
        \draw[CliqueEdge](2)edge[]node[CliqueLabel]
            {\begin{math}4\end{math}}(3);
        \draw[CliqueEdge](2)edge[]node[CliqueLabel]
            {\begin{math}1\end{math}}(4);
        \draw[CliqueEdge](3)edge[]node[CliqueLabel]
            {\begin{math}2\end{math}}(4);
        \draw[CliqueEmptyEdge](4)edge[]node[]{}(5);
        \draw[CliqueEdge](5)edge[]node[CliqueLabel]
            {\begin{math}3\end{math}}(6);
        \draw[CliqueEdge](5)edge[]node[CliqueLabel]
            {\begin{math}3\end{math}}(9);
        \draw[CliqueEdge](5)edge[]node[CliqueLabel]
            {\begin{math}1\end{math}}(10);
        \draw[CliqueEmptyEdge](6)edge[]node[]{}(7);
        \draw[CliqueEdge](6)edge[]node[CliqueLabel]
            {\begin{math}2\end{math}}(9);
        \draw[CliqueEdge](7)edge[]node[CliqueLabel]
            {\begin{math}1\end{math}}(8);
        \draw[CliqueEmptyEdge](8)edge[]node[]{}(9);
        \draw[CliqueEmptyEdge](9)edge[]node[]{}(10);
    \end{tikzpicture}
    \end{split}
    \quad \longleftrightarrow \quad
    \begin{split}
    \begin{tikzpicture}[xscale=.4,yscale=.18]
        \node[Leaf](0)at(0.00,-6.00){};
        \node[Leaf](11)at(9.00,-12.00){};
        \node[Leaf](12)at(10.00,-12.00){};
        \node[Leaf](14)at(12.00,-6.00){};
        \node[Leaf](2)at(1.00,-9.00){};
        \node[Leaf](4)at(3.00,-9.00){};
        \node[Leaf](5)at(4.00,-6.00){};
        \node[Leaf](7)at(6.00,-9.00){};
        \node[Leaf](9)at(8.00,-12.00){};
        \node[Node](1)at(2.00,-3.00){};
        \node[Node](10)at(9.00,-9.00){};
        \node[Node](13)at(11.00,-3.00){};
        \node[Node](3)at(2.00,-6.00){};
        \node[Node](6)at(5.00,0.00){};
        \node[Node](8)at(7.00,-6.00){};
        \draw[Edge](0)edge[]node[EdgeLabel]{\begin{math}1\end{math}}(1);
        \draw[Edge](1)edge[]node[EdgeLabel]{\begin{math}2\end{math}}(6);
        \draw[Edge](10)edge[]node[EdgeLabel]{\begin{math}2\end{math}}(8);
        \draw[Edge](11)edge[]node[EdgeLabel]{\begin{math}1\end{math}}(10);
        \draw[Edge](12)edge[]node[EdgeLabel]{\begin{math}0\end{math}}(10);
        \draw[Edge](13)edge[]node[EdgeLabel]{\begin{math}1\end{math}}(6);
        \draw[Edge](14)edge[]node[EdgeLabel]{\begin{math}0\end{math}}(13);
        \draw[Edge](2)edge[]node[EdgeLabel]{\begin{math}4\end{math}}(3);
        \draw[Edge](3)edge[]node[EdgeLabel]{\begin{math}1\end{math}}(1);
        \draw[Edge](4)edge[]node[EdgeLabel]{\begin{math}2\end{math}}(3);
        \draw[Edge](5)edge[]node[EdgeLabel]{\begin{math}0\end{math}}(1);
        \draw[Edge](7)edge[]node[EdgeLabel]{\begin{math}3\end{math}}(8);
        \draw[Edge](8)edge[]node[EdgeLabel]{\begin{math}3\end{math}}(13);
        \draw[Edge](9)edge[]node[EdgeLabel]{\begin{math}0\end{math}}(10);
        \node(r)at(5.00,3){};
        \draw[Edge](r)edge[]node[EdgeLabel]{\begin{math}1\end{math}}(6);
    \end{tikzpicture}
    \end{split}\,.
\end{equation}

By seeing the elements of $\NC\Mca$ as $\Mca$-dual trees, we can
rephrase the partial composition of this operad as follows. If $\Sfr$
and $\Tfr$ are two $\Mca$-dual trees and $i$ is a valid integer, the
tree $\Sfr \circ_i \Tfr$ is computed by grafting the root of $\Tfr$ to
the $i$-th leaf of $\Sfr$. Then, by denoting by $b$ the label of the
edge adjacent to the root of $\Tfr$ and by $a$ the label of the edge
adjacent to the $i$-th leaf of $\Sfr$, we have two cases to consider,
depending on the value of $c := a \Op b$. If $c \ne \Unit_\Mca$, we
label the edge connecting $\Sfr$ and $\Tfr$ by $c$. Otherwise, when
$c = \Unit_\Mca$, we contract the edge connecting $\Sfr$ and $\Tfr$ by
merging the root of $\Tfr$ and the father of the $i$-th leaf of $\Sfr$.
For instance, in $\NC\N_3$, we have
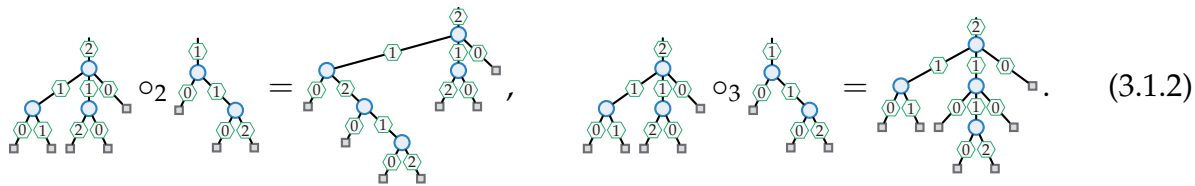
\vspace{-1em}
\begin{equation}
    \begin{split}
    \begin{tikzpicture}[xscale=.25,yscale=.2]
        \node[Leaf](0)at(0.00,-5.33){};
        \node[Leaf](2)at(2.00,-5.33){};
        \node[Leaf](4)at(3.00,-5.33){};
        \node[Leaf](6)at(5.00,-5.33){};
        \node[Leaf](7)at(6.00,-2.67){};
        \node[Node](1)at(1.00,-2.67){};
        \node[Node](3)at(4.00,0.00){};
        \node[Node](5)at(4.00,-2.67){};
        \draw[Edge](0)edge[]node[EdgeLabel]{\begin{math}0\end{math}}(1);
        \draw[Edge](1)edge[]node[EdgeLabel]{\begin{math}1\end{math}}(3);
        \draw[Edge](2)edge[]node[EdgeLabel]{\begin{math}1\end{math}}(1);
        \draw[Edge](4)edge[]node[EdgeLabel]{\begin{math}2\end{math}}(5);
        \draw[Edge](5)edge[]node[EdgeLabel]{\begin{math}1\end{math}}(3);
        \draw[Edge](6)edge[]node[EdgeLabel]{\begin{math}0\end{math}}(5);
        \draw[Edge](7)edge[]node[EdgeLabel]{\begin{math}0\end{math}}(3);
        \node(r)at(4.00,2.75){};
        \draw[Edge](r)edge[]node[EdgeLabel]{\begin{math}2\end{math}}(3);
    \end{tikzpicture}
    \end{split}
    \circ_2
    \begin{split}
    \begin{tikzpicture}[xscale=.25,yscale=.3]
        \node[Leaf](0)at(0.00,-1.67){};
        \node[Leaf](2)at(2.00,-3.33){};
        \node[Leaf](4)at(4.00,-3.33){};
        \node[Node](1)at(1.00,0.00){};
        \node[Node](3)at(3.00,-1.67){};
        \draw[Edge](0)edge[]node[EdgeLabel]{\begin{math}0\end{math}}(1);
        \draw[Edge](2)edge[]node[EdgeLabel]{\begin{math}0\end{math}}(3);
        \draw[Edge](3)edge[]node[EdgeLabel]{\begin{math}1\end{math}}(1);
        \draw[Edge](4)edge[]node[EdgeLabel]{\begin{math}2\end{math}}(3);
        \node(r)at(1.00,2){};
        \draw[Edge](r)edge[]node[EdgeLabel]{\begin{math}1\end{math}}(1);
    \end{tikzpicture}
    \end{split}
    =
    \begin{split}
    \begin{tikzpicture}[xscale=.25,yscale=.2]
        \node[Leaf](0)at(0.00,-4.80){};
        \node[Leaf](10)at(9.00,-4.80){};
        \node[Leaf](11)at(10.00,-2.40){};
        \node[Leaf](2)at(2.00,-7.20){};
        \node[Leaf](4)at(4.00,-9.60){};
        \node[Leaf](6)at(6.00,-9.60){};
        \node[Leaf](8)at(7.00,-4.80){};
        \node[Node](1)at(1.00,-2.40){};
        \node[Node](3)at(3.00,-4.80){};
        \node[Node](5)at(5.00,-7.20){};
        \node[Node](7)at(8.00,0.00){};
        \node[Node](9)at(8.00,-2.40){};
        \draw[Edge](0)edge[]node[EdgeLabel]{\begin{math}0\end{math}}(1);
        \draw[Edge](1)edge[]node[EdgeLabel]{\begin{math}1\end{math}}(7);
        \draw[Edge](10)edge[]node[EdgeLabel]{\begin{math}0\end{math}}(9);
        \draw[Edge](11)edge[]node[EdgeLabel]{\begin{math}0\end{math}}(7);
        \draw[Edge](2)edge[]node[EdgeLabel]{\begin{math}0\end{math}}(3);
        \draw[Edge](3)edge[]node[EdgeLabel]{\begin{math}2\end{math}}(1);
        \draw[Edge](4)edge[]node[EdgeLabel]{\begin{math}0\end{math}}(5);
        \draw[Edge](5)edge[]node[EdgeLabel]{\begin{math}1\end{math}}(3);
        \draw[Edge](6)edge[]node[EdgeLabel]{\begin{math}2\end{math}}(5);
        \draw[Edge](8)edge[]node[EdgeLabel]{\begin{math}2\end{math}}(9);
        \draw[Edge](9)edge[]node[EdgeLabel]{\begin{math}1\end{math}}(7);
        \node(r)at(8.00,2.5){};
        \draw[Edge](r)edge[]node[EdgeLabel]{\begin{math}2\end{math}}(7);
    \end{tikzpicture}
    \end{split}\,,
    \qquad
    \begin{split}
    \begin{tikzpicture}[xscale=.25,yscale=.2]
        \node[Leaf](0)at(0.00,-5.33){};
        \node[Leaf](2)at(2.00,-5.33){};
        \node[Leaf](4)at(3.00,-5.33){};
        \node[Leaf](6)at(5.00,-5.33){};
        \node[Leaf](7)at(6.00,-2.67){};
        \node[Node](1)at(1.00,-2.67){};
        \node[Node](3)at(4.00,0.00){};
        \node[Node](5)at(4.00,-2.67){};
        \draw[Edge](0)edge[]node[EdgeLabel]{\begin{math}0\end{math}}(1);
        \draw[Edge](1)edge[]node[EdgeLabel]{\begin{math}1\end{math}}(3);
        \draw[Edge](2)edge[]node[EdgeLabel]{\begin{math}1\end{math}}(1);
        \draw[Edge](4)edge[]node[EdgeLabel]{\begin{math}2\end{math}}(5);
        \draw[Edge](5)edge[]node[EdgeLabel]{\begin{math}1\end{math}}(3);
        \draw[Edge](6)edge[]node[EdgeLabel]{\begin{math}0\end{math}}(5);
        \draw[Edge](7)edge[]node[EdgeLabel]{\begin{math}0\end{math}}(3);
        \node(r)at(4.00,2.75){};
        \draw[Edge](r)edge[]node[EdgeLabel]{\begin{math}2\end{math}}(3);
    \end{tikzpicture}
    \end{split}
    \circ_3
    \begin{split}
    \begin{tikzpicture}[xscale=.25,yscale=.3]
        \node[Leaf](0)at(0.00,-1.67){};
        \node[Leaf](2)at(2.00,-3.33){};
        \node[Leaf](4)at(4.00,-3.33){};
        \node[Node](1)at(1.00,0.00){};
        \node[Node](3)at(3.00,-1.67){};
        \draw[Edge](0)edge[]node[EdgeLabel]{\begin{math}0\end{math}}(1);
        \draw[Edge](2)edge[]node[EdgeLabel]{\begin{math}0\end{math}}(3);
        \draw[Edge](3)edge[]node[EdgeLabel]{\begin{math}1\end{math}}(1);
        \draw[Edge](4)edge[]node[EdgeLabel]{\begin{math}2\end{math}}(3);
        \node(r)at(1.00,2){};
        \draw[Edge](r)edge[]node[EdgeLabel]{\begin{math}1\end{math}}(1);
    \end{tikzpicture}
    \end{split}
    =
    \begin{split}
    \begin{tikzpicture}[xscale=.25,yscale=.2]
        \node[Leaf](0)at(0.00,-5.50){};
        \node[Leaf](10)at(8.00,-2.75){};
        \node[Leaf](2)at(2.00,-5.50){};
        \node[Leaf](4)at(3.00,-5.50){};
        \node[Leaf](6)at(4.00,-8.25){};
        \node[Leaf](8)at(6.00,-8.25){};
        \node[Leaf](9)at(7.00,-5.50){};
        \node[Node](1)at(1.00,-2.75){};
        \node[Node](3)at(5.00,0.00){};
        \node[Node](5)at(5.00,-2.75){};
        \node[Node](7)at(5.00,-5.50){};
        \draw[Edge](0)edge[]node[EdgeLabel]{\begin{math}0\end{math}}(1);
        \draw[Edge](1)edge[]node[EdgeLabel]{\begin{math}1\end{math}}(3);
        \draw[Edge](10)edge[]node[EdgeLabel]{\begin{math}0\end{math}}(3);
        \draw[Edge](2)edge[]node[EdgeLabel]{\begin{math}1\end{math}}(1);
        \draw[Edge](4)edge[]node[EdgeLabel]{\begin{math}0\end{math}}(5);
        \draw[Edge](5)edge[]node[EdgeLabel]{\begin{math}1\end{math}}(3);
        \draw[Edge](6)edge[]node[EdgeLabel]{\begin{math}0\end{math}}(7);
        \draw[Edge](7)edge[]node[EdgeLabel]{\begin{math}1\end{math}}(5);
        \draw[Edge](8)edge[]node[EdgeLabel]{\begin{math}2\end{math}}(7);
        \draw[Edge](9)edge[]node[EdgeLabel]{\begin{math}0\end{math}}(5);
        \node(r)at(5.00,2.5){};
        \draw[Edge](r)edge[]node[EdgeLabel]{\begin{math}2\end{math}}(3);
    \end{tikzpicture}
    \end{split}\,.
\end{equation}

Let $\Triangles_\Mca$ be the set of all $\Mca$-cliques of arity $2$. We
call such cliques {\em $\Mca$-triangles}.

\begin{Proposition} \label{prop:generating_set_Cro0_M}
    Let $\Mca$ be a unitary magma. The set $\Triangles_\Mca$ is the
    unique minimal generating set of~$\NC\Mca$.
\end{Proposition}

\begin{Proposition} \label{prop:Hilbert_series_Cro0_M}
    Let $\Mca$ be a finite unitary magma and $m$ be its cardinality.
    The Hilbert series
    $\Hilbert_{\NC\Mca}(t)$ of $\NC\Mca$ satisfies
    \begin{equation}
        t + \left(m^3 - 2m^2 + 2m - 1\right)t^2
        + \left(2m^2t - 3mt + 2t - 1\right) \Hilbert_{\NC\Mca}(t)
        + \left(m - 1\right) \Hilbert_{\NC\Mca}(t)^2
        = 0.
    \end{equation}
\end{Proposition}

From this result, together with classical arguments involving
Narayana numbers, we obtain that for all $n \geq 2$,
\vspace{-.5em}
\begin{equation}
    \dim \NC\Mca(n) =
    \sum_{0 \leq k \leq n - 2}
        m^{n + k + 1} (m - 1)^{n - k - 2} \:
        \frac{1}{k + 1} \binom{n - 2}{k} \binom{n - 1}{k}.
\end{equation}
For instance, when $\# \Mca = 2$, the dimensions of $\NC\Mca$ begin by
$1$, $8$, $48$, $352$, $2880$ (Sequence~\OEIS{A054726}, except for the
first terms).

\subsection{Presentation, Koszulity, and Koszul dual}
In what follows, $\Mca$-triangles
\begin{math}
    \Pfr =
    \begin{tikzpicture}[scale=0.4,baseline=.00em]
        \node[CliquePoint](1)at(-0.87,-0.50){};
        \node[CliquePoint](2)at(-0.00,1.00){};
        \node[CliquePoint](3)at(0.87,-0.50){};
        \draw[CliqueEdge](1)edge[]node[CliqueLabel]
             {\begin{math}\Pfr_1\end{math}}(3);
        \draw[CliqueEdge](2)edge[]node[CliqueLabel]
            {\begin{math}\Pfr_3\end{math}}(3);
        \draw[CliqueEdge](1)edge[]node[CliqueLabel]
            {\begin{math}\Pfr_2\end{math}}(2);
    \end{tikzpicture}
\end{math}
are denoted by words $\Pfr_1\Pfr_2\Pfr_3 \in \Mca^3$.

\begin{Theorem} \label{thm:presentation_Cro0_M}
    Let $\Mca$ be a finite unitary magma. The operad $\NC\Mca$ is
    binary, quadratic, Koszul, and admits the presentation
    $\NC\Mca \simeq \Free(\Triangles_\Mca)/_{\langle\Rel\rangle}$, where
    $\Rel$ is the space of relations generated by
    \begin{subequations}
    \begin{equation}
        \Pfr_1\Pfr_2\Pfr_3
        \circ_1
        \Qfr_1\Qfr_2\Qfr_3
        -
        \Pfr_1\Rfr_2\Pfr_3
        \circ_1
        \Rfr_1\Qfr_2\Qfr_3,
        \qquad
        \mbox{if } \Pfr_2 \Op \Qfr_1 = \Rfr_2 \Op \Rfr_1 \ne \Unit_\Mca,
    \end{equation}
    \vspace{-1em}
    \begin{equation}
        \Pfr_1\Pfr_2\Pfr_3
        \circ_1
        \Qfr_1\Qfr_2\Qfr_3
        -
        \Pfr_1\Qfr_2\Rfr_3
        \circ_2
        \Rfr_1\Qfr_3\Pfr_3,
        \qquad
        \mbox{if } \Pfr_2 \Op \Qfr_1 = \Rfr_3 \Op \Rfr_1 = \Unit_\Mca,
    \end{equation}
    \vspace{-1em}
    \begin{equation}
        \Pfr_1\Pfr_2\Pfr_3
        \circ_2
        \Qfr_1\Qfr_2\Qfr_3
        -
        \Pfr_1\Pfr_2\Rfr_3
        \circ_2
        \Rfr_1\Qfr_2\Qfr_3,
        \qquad
        \mbox{if } \Pfr_3 \Op \Qfr_1 = \Rfr_3 \Op \Rfr_1 \ne \Unit_\Mca,
    \end{equation}
    \end{subequations}
    where
    \begin{math}
        \Pfr_1, \Pfr_2, \Pfr_3, \Qfr_1, \Qfr_2, \Qfr_3,
        \Rfr_1, \Rfr_2, \Rfr_3 \in \Mca
    \end{math}.
\end{Theorem}
\begin{proof}
    The proof is long, technical, but classical and uses techniques from
    rewriting theory~\cite{BN98}. It consists in defining a rewrite rule
    $\to$ from $\Rel$ on syntax trees of $\Mca$-triangles, by showing
    that $\to$ is convergent, and prove that $\to$ admits as many normal
    forms as basis elements of $\NC\Mca$ of arity $n$ for all~$n \geq 1$.
    The fact that $\NC\Mca$ is Koszul is a consequence of the existence
    of such a rewrite rule $\to$ (see~\cite{DK10,Hof10}).
\end{proof}

We can now compute a presentation of the Koszul dual $\NC\Mca^!$ of
$\NC\Mca$ from the presentation of $\NC\Mca$ provided by
Theorem~\ref{thm:presentation_Cro0_M}.

\begin{Proposition} \label{prop:presentation_Cro0_M_dual}
    Let $\Mca$ be a finite unitary magma. The operad $\NC\Mca^!$ admits
    the presentation
    \begin{math}
        \NC\Mca^! \simeq
        \Free(\Triangles_\Mca)/_{\left\langle\Rel^\perp\right\rangle},
    \end{math}
    where $\Rel^\perp$ is the space of relations generated by
    \begin{subequations}
    \begin{equation}\begin{split}\end{split}
        \sum_{\substack{
            \Pfr_2, \Qfr_1 \in \Mca,
            \Pfr_2 \Op \Qfr_1 = \delta
        }}
        \enspace
        \Pfr_1\Pfr_2\Pfr_3 \circ_1 \Qfr_1\Qfr_2\Qfr_3,
        \qquad
        \mbox{where } \Pfr_1, \Pfr_3, \Qfr_2, \Qfr_3 \in \Mca, \delta
        \in \Mca \setminus \{\Unit_\Mca\},
    \end{equation}
    \vspace{-1em}
    \begin{equation}\begin{split}\end{split}
        \sum_{\substack{
            \Pfr_2, \Qfr_1 \in \Mca,
            \Pfr_2 \Op \Qfr_1 = \Unit_\Mca
        }}
        \enspace
        \Pfr_1\Pfr_2\Pfr_3 \circ_1 \Qfr_1\Qfr_2\Qfr_3
        -
        \Pfr_1\Qfr_2\Pfr_2 \circ_2 \Qfr_1\Qfr_3\Pfr_3,
        \qquad
        \mbox{where } \Pfr_1, \Pfr_3, \Qfr_2, \Qfr_3 \in \Mca,
    \end{equation}
    \vspace{-1em}
    \begin{equation}\begin{split}\end{split}
        \sum_{\substack{
            \Pfr_3, \Qfr_1 \in \Mca,
            \Pfr_3 \Op \Qfr_1 = \delta
        }}
        \enspace
        \Pfr_1\Pfr_2\Pfr_3 \circ_2 \Qfr_1\Qfr_2\Qfr_3,
        \qquad
        \mbox{where } \Pfr_1, \Pfr_2, \Qfr_2, \Qfr_3 \in \Mca, \delta
        \in \Mca \setminus \{\Unit_\Mca\}.
    \end{equation}
    \end{subequations}
\end{Proposition}

\begin{Proposition} \label{prop:Hilbert_series_Cro0_M_dual}
    Let $\Mca$ be a finite unitary magma and $m$ be its cardinality. The
    Hilbert series $\Hilbert_{\NC\Mca^!}(t)$ of $\NC\Mca^!$ satisfies
    \begin{equation}
        t + (m - 1)t^2
        + \left(2m^2t - 3mt + 2t -1\right)\Hilbert_{\NC\Mca^!}(t)
        + \left(m^3 - 2m^2 + 2m - 1\right)\Hilbert_{\NC\Mca^!}(t)^2 = 0.
    \end{equation}
    \noindent This is the generating series of all noncrossing
    configurations where all edges and bases are labeled by pairs
    $(a, a) \in \Mca^2$, and all solid diagonals are labeled by pairs
    $(a, b) \in \Mca^2$ where~$a \ne b$.
\end{Proposition}

Proposition~\ref{prop:Hilbert_series_Cro0_M_dual} hence provides a
combinatorial description of the elements of $\NC\Mca^!$. For instance,
when $\# \Mca = 2$, the dimensions of $\NC\Mca^!$ begin by $1$, $8$,
$80$, $992$, $13760$ (Sequence~\OEIS{A234596}). It is worthwhile to
observe that the dimensions of $\NC\Mca^!$ in this case are the ones of
the operad~$\BNC$~\cite{CG14} (see
Section~\ref{subsec:constructing_existing_operads}).


\bibliographystyle{plain}
\bibliography{Bibliography}

\end{document}